\documentclass[draft]{amsart}

\newtheorem{thm}{Theorem}[section]

\newtheorem{lem}[thm]{Lemma}

\theoremstyle{remark}
\newtheorem{rem}[thm]{Remark}

\theoremstyle{definition}
\newtheorem{dfn}[thm]{Definition}

\numberwithin{equation}{section}
\numberwithin{thm}{section}

\begin{document}






\title[Commutators of multilinear strongly singular integrals]{Commutators of multilinear strongly singular integrals on non-homogeneous metric measure spaces}


%

\author{Hailian Wang and Rulong Xie$^{\ast}$}
\address{Hailian Wang, Department of Mathematics, Chaohu University,
Hefei 238000, China} \email{hailianmath@163.com}
\address{Rulong Xie ($\ast$ Corresponding author), 1. Department
of Mathematics, Chaohu University, Hefei 238000,
China; 2. School of Mathematical Sciences, University of
Science and Technology of China, Hefei 230026, China}\email{xierl@mail.ustc.edu.cn}


%

%

\begin{abstract}
Let $(X,d,\mu)$ denotes non-homogeneous metric measure space satisfying geometrically doubling and the upper doubling measure condition. In this paper, the boundedness in Lebesgue spaces for two kinds of commutators, which are iterated commutators and commutators in summation form, generated by multilinear strongly singular integral operators with RBMO$(\mu)$ function on non-homogeneous metric measure spaces $(X,d,\mu)$ is obtained.
\end{abstract}

\maketitle

\section{Introduction}

It is well known that Hyt\"{o}nen \cite{H} introduced non-homogeneous metric measure spaces $(X,d,\mu)$,
 satisfying geometrically doubling and the upper doubling measure condition (see
Definitions \ref{def1.1} and \ref{def1.2}). The non-homogeneous metric measure space includes both the
homogeneous spaces and non-doubling measure spaces as special cases. We recall that the
homogeneous space is a space equipped with a
non-negative doubling  measure $\mu$, where $\mu$ is said to satisfy the
doubling condition if there exists a constant $C>0$ such that
$\mu(B(x,2r))\leq C\mu(B(x,r))$ for all $x\in \rm{supp} \mu$ and
$r>0$. The non-doubling measure space is a space equipped with a non-negative
measure $\mu$, where $\mu$ only needs to satisfy the polynomial growth condition,
i.e., for all $x \in R^{n}$ and $r>0$, there exist a constant $C_0>0$ and
$k\in(0, n]$ such that
\begin{equation}\label{1.1}
\mu(B(x, r)) \leq C_{0}r^{k},
\end{equation}
where $B(x, r)= \{y \in R^{n}: |y-x| < r\}$. There are many important results in non-doubling measure spaces (see \cite{HMY1}, \cite{HMY2}, \cite{NTV}, \cite{T1}, \cite{T2}, \cite{XS}, \cite{X1} and the references therein). And
the analysis on non-doubling measures has important applications in
solving the long-standing open Painlev$\acute{e}$'s problem (see
\cite{T1}).

Next let us recall some results on non-homogeneous metric measure spaces. Hyt\"{o}nen et al. \cite{HYY}
and Bui and Duong \cite{BD1} independently introduced the atomic Hardy
space $H^{1}(\mu)$ and proved that the dual space of $H^{1}(\mu)$ is
RBMO$(\mu)$. In \cite{BD1}, the authors also proved that
Calder\'{o}n-Zygmund operator and commutators generated by
Calder\'{o}n-Zygmund operators and RBMO function are bounded in
$L^{p}(\mu)(1<p<+\infty)$. Recently, some equivalent
characterizations are established by Liu et al. \cite{LYY} for the
boundedness of Calder\'{o}n-Zygmund operators on $L^{p}(\mu)(1<p<+\infty)$. In \cite{FYY1}, Fu et al. established boundedness of
multilinear commutators of Calder\'{o}n-Zygmund operators on Orlicz
spaces on non-homogeneous spaces. The more results on non-homogeneous metric measure spaces have been obtained (see
\cite{CL}, \cite{FYY2}, \cite{GXX}, \cite{HM}-\cite{LY3}, \cite{LMY}, \cite{YYH} and the references therein).

   The theory on multilinear singular integral
operators has been considered by some researchers. In \cite{CM}, Coifman
and Meyer firstly established the theory of bilinear
Calder\'{o}n-Zygmund operators. Later, Grafakos and Torres \cite{GT1,GT2}
established the boundedness of multilinear singular integral on the
product Lebesgue spaces and Hardy spaces. The boundedness of
multilinear singular integrals and commutators on non-doubling
measure spaces $(R^{n},\mu)$ was established by Xu \cite{X1,X2}.
Weighted norm inequalities for multilinear Calder\'{o}n-Zygmund
operators on non-homogeneous metric measure spaces were also
constructed in \cite{HMY3}. Boundedness for commutators of multilinear Calder\'{o}n-Zygmund
operators and multilinear fractional integral
operators on non-homogeneous metric measure spaces was also
established in \cite{XGZ,GXX}. Zheng and Tao \cite{ZT} established the boundedness for iterated commutators of multilinear singular integrals of Dini's type on non-homogeneous metric measure spaces.

The introduction of strongly singular integral operator is motivated by a class of multiplier
operators whose symbol is given by $e^{i|\xi|\alpha} /|\xi|\beta$ away from the origin, where $0<\alpha<1$ and
$\beta>0$. Fefferman and Stein \cite{FS} enlarged the multiplier operators onto a class of convolution
operators. Coifman \cite{C} also considered a related class of operators for $n = 1$.
The strongly singular non-convolution operators were introduced and researched by Alvarez and Milman \cite{AM1,AM2},
whose properties are similar to those of Calder\'{o}n-Zygmund operators, but the kernel is more
singular near the diagonal than those of the standard case. Furthermore, Lin and Lu (\cite{L}-\cite{LL3}) obtained the boundedness for strongly
singular integral and its commutators on Lebesgue spaces, Morrey spaces and Hardy spaces. Recently, we \cite{WX} established the boundedness for multilinear strongly singular integral operators on non-homogeneous metric measure spaces.

In this paper, two kinds of commutators generated by multilinear strongly singular integral operators
and RBMO$(\mu)$ function on non-homogeneous metric measure spaces are introduced. And we will  prove that they are bounded in $m$-multiple Lebesgue
spaces, provided that multilinear
strongly singular integral is bounded from $m$-multiple $L^1(\mu)\times \ldots
\times L^1(\mu)$ to $L^{1/m,\infty}(\mu)$, where $L^{p}(\mu)$ and
$L^{p,\infty}(\mu)$ denote the Lebesgue space and weak Lebesgue
space respectively.

Before stating the main results, let us first recall some
notations and definitions.
\begin{dfn}\cite{H}\label{def1.1}
A metric space $(X,d)$
is called geometrically doubling if there exists some $N_{0}\in
\mathbf{N}$ such that, for any ball $B(x,r)\subset X$, there exists
a finite ball covering $\{B(x_i,r/2)\}_i$ of $B(x,r)$ such that the
cardinality of this covering is at most $N_0$.
 \end{dfn}

\begin{dfn}\cite{H}\label{def1.2}
A metric measure space
$(X,d,\mu)$ is said to be upper doubling if $\mu$ is a Borel measure
on $X$ and there exists a dominating function $\lambda : X
\times(0,+\infty) \rightarrow (0,+\infty)$ and a constant
$C_{\lambda} >0$ such that for each $x\in X, r \longmapsto \lambda(x,r)$ is
non-decreasing, and for all $x\in X, r >0$,
 \begin{equation}\label{1.2}
 \mu(B(x, r))\leq \lambda(x, r)\leq C_{\lambda}\lambda(x, r/2).
 \end{equation}
\end{dfn}

\begin{rem}(i)\ A space of homogeneous type is a special case of upper
doubling spaces, where one can take the dominating function
$\lambda(x, r)= \mu(B(x,r))$. On the other hand, a metric space
$(X,d,\mu)$ satisfying the polynomial growth condition (\ref{1.1}) (in
particular, $(X,d,\mu)=(R^n, |\cdot|, \mu)$ with $\mu$
satisfying (\ref{1.1}) for some $k\in (0, n])$ is also an upper doubling
measure space if we take $\lambda(x, r)=Cr^{k}$.

(ii)\ Let$(X,d,\mu)$ be an upper doubling space and $\lambda$ be a
dominating function on $X \times(0,+\infty)$ as in Definition 1.2. In
\cite{HYY}, it was shown that there exists another dominating function
$\tilde{\lambda}$ such that for all $x, y \in X$ with $d(x, y)\leq
r$,
\begin{equation}\label{1.3}
\tilde{\lambda}(x, r)\leq \tilde{C}\tilde{\lambda}(y, r).
\end{equation}
 Thus, we assume that $\lambda$ always satisfies (\ref{1.3}) in this paper.
\end{rem}

\begin{dfn}\cite{BD1}
 Let $1<\alpha,\beta<+\infty$. A ball $B\subset X$ is called $(\alpha,
\beta)$-doubling if $\mu(\alpha B)\leq \beta \mu (B)$.
\end{dfn}

\begin{rem}
As pointed in Lemma 2.3 of \cite{BD1},  there exist plenty of doubling
balls with small radii and with large radii. In the rest of this
paper, unless $\alpha$ and $\beta$ are specified otherwise, by an
$(\alpha,\beta)$ doubling ball we mean a $(6,\beta_{0})$-doubling
with a fixed number $\beta_0 >\max\{C_{\lambda}^{3\log_{2}6},
6^{n}\}$, where $n=\log_{2}N_{0}$ is viewed as a geometric dimension
of the space.
\end{rem}

\begin{dfn}\label{def1.6}
 A kernel $K(\cdot,\cdots,\cdot)\in L_{loc}^{1}((X)^{m+1}\backslash\{(x,y_{1}\cdots,y_{j},\cdots,y_{m}):x=y_{1}=\cdots=y_{j}=\cdots=y_{m}\})$
 is called an $m$-linear strongly singular integral kernel if it satisfies:

(i)\begin{equation}\label{1.4}
 |K(x,y_{1},\cdots,y_{j},\cdots, y_{m})|\leq
C\biggl[\sum_{j=1}^{m}\lambda(x,d(x,y_{j}))\biggr]^{-m}
 \end{equation}
for all $(x,y_{1}\cdots,y_{j},\cdots,y_{m})\in (X)^{m+1}$ with
$x\neq y_{j}$ for some $j$.

 (ii) There exists $0<\alpha<1$ and $0<\delta\leq 1$ such that

\begin{equation}\label{1.5}
\begin{split}
&|K(x,y_{1},\cdots,y_{j},\cdots,y_{m})-K(x',y_{1},\cdots,y_{j},\cdots,y_{m})|\\
\leq
&\frac{Cd(x,x')^{\delta}}{\biggl[\sum\limits_{j=1}^{m}d(x,y_{j})\biggr]^{\delta/\alpha}\biggl[\sum\limits_{j=1}^{m}\lambda(x,d(x,y_{j}))\biggr]^{m}},
\end{split}
\end{equation}
provided that $Cd(x,x')^{\alpha}\leq \max\limits_{1\leq j\leq m}d(x,y_{j})$ and
for each $j$,
 \begin{equation}\label{1.6}
 \begin{split}
&|K(x,y_{1},\cdots,y_{j},\cdots,y_{m})-K(x,y_{1},\cdots,y'_{j},\cdots,y_{m})|\\
 &\leq\frac{Cd(y_{j},y'_{j})^{\delta}}{\biggl[\sum\limits_{j=1}^{m}d(x,y_{j})\biggr]^{\delta/\alpha}
 \biggl[\sum\limits_{j=1}^{m}\lambda(x,d(x,y_{j}))\biggr]^{m}},
\end{split}
\end{equation}
provided that $Cd(y_{j},y'_{j})^{\alpha}\leq \max\limits_{1\leq j\leq
m}d(x,y_{j})$, where $C$ is a positive constant.

A multilinear operator $T$ is called a multilinear
strongly singular integral operator with the above
kernel $K$ satisfying $(\ref{1.4})$, $(\ref{1.5})$ and $(\ref{1.6})$ if, for $f_{1},\cdots
f_{m}$ are $L^{\infty}$ functions with bounded support and $x\notin
\bigcap_{j=1}^{m}\text{supp}f_{j}$,
 \begin{equation}\label{1.7}
 T(f_{1},\cdots f_{m})(x)=\int_{X^{m}}K(x,y_{1},\cdots y_{m})f_{1}(y_{1})\cdots
 f_{m}(y_{m})d\mu(y_{1})\cdots d\mu(y_{m}).
\end{equation}
\end{dfn}

\begin{dfn}\cite{BD1} For any two balls $B\subset Q$, let $N_{B,Q}$ be the smallest integer
satisfying $6^{N_{B,Q}}r_{B}\geq r_Q$, denote
\begin{equation}
K_{B,Q} = 1+\sum_{k=1}^{N_{B,Q}}
\frac{\mu(6^{k}B)}{\lambda(x_B,6^{k}r_{B})},
\end{equation}
where $x_B$ and $r_{B}$ denote center and radius of ball $B$ respectively.
\end{dfn}

\begin{dfn}
\cite{BD1} Let $\rho>1$ be some
fixed constant. A function $b\in L_{loc}^{1}(\mu)$ is said to belong
to RBMO$(\mu)$ if there exists a constant $C
>0$ such that for any ball $B$,
\begin{equation}\label{1.9}
\frac{1}{\mu(\rho B)}\int_{B}|b(x)-m_{\widetilde{B}}(b)|d\mu(x)\leq C,
\end{equation}
and for any two doubling balls $B\subset Q$,
 \begin{equation}\label{1.10}
  |m_{B}(b)-m_{Q}(b)|\leq CK_{B,Q},
\end{equation}
where $\widetilde{B}$ is the smallest $(\alpha,\beta)$-doubling ball of
the form $6^{k}B$ with $k\in {\mathbf{N}}\bigcup\{0\}$, and
$m_{\widetilde{B}}(b)$ is the mean value of $b$ on $\widetilde{B}$,
namely,
$$m_{\widetilde{B}}(b)=\frac{1}{\mu(\widetilde{B})}\int_{\widetilde{B}}b(x)d\mu(x).$$
The minimal constant $C$ appearing in (\ref{1.9}) and (\ref{1.10}) is defined to be
the RBMO$(\mu)$ norm of $b$ and denoted by $||b||_{\ast}$.
\end{dfn}

Let us introduce the definitions of two kinds of commutators, one can refer to the references \cite{PPTT,PT}.

\begin{dfn}
Iterated commutators generated
by multilinear strongly singular integral $T$ with $\vec{b}=(b_{1},\cdots,b_{m})\in
\text{RBMO}(\mu)^{m}$ is defined by
\begin{equation}
T_{\prod \vec{b}}(f_1,\cdots,f_m)= [b_1, [b_2,... [b_{m-1}, [b_m, T]_{m}]_{m-1}\cdots]_{2}]_{1}(f_1,\cdots,f_m).
\end{equation}
\end{dfn}

In particular, when $m=2$, we obtain
\begin{equation}\label{1.12}
 \begin{split}
[b_{1},b_{2},T](f_{1},f_{2})(x)=&b_{1}(x)b_{2}(x)T(f_{1},f_{2})(x)-b_{1}(x)T(f_{1},b_{2}f_{2})(x)\\
&-b_{2}(x)T(b_{1}f_{1},f_{2})(x)+T(b_{1}f_{1},b_{2}f_{2})(x).
\end{split}
\end{equation}
Also, we define $[b_{1},T]$ and $[b_{2},T]$ as follows respectively.
\begin{equation}[b_{1},T](f_{1},f_{2})(x)=b_{1}(x)T(f_{1},f_{2})(x)-T(b_{1}f_{1},f_{2})(x),\end{equation}
\begin{equation}[b_{2},T](f_{1},f_{2})(x)=b_{2}(x)T(f_{1},f_{2})(x)-T(f_{1},b_{2}f_{2})(x).\end{equation}

\begin{dfn}
The commutators in the summation form generated
by multilinear strongly singular integral $T$ with $\vec{b}=(b_{1},\cdots,b_{m})\in
\text{RBMO}(\mu)^{m}$ is defined by
\begin{equation}\label{1.51}
T_{\sum \vec{b}}(f_{1},\cdots,f_{m})=\sum_{j=1}^{m}T_{b_{j}}^{j}(f_{1},\cdots,f_{m}),
\end{equation}
where
\begin{equation*}
T_{b_{j}}^{j}(f_{1},\cdots,f_{m})=b_{j}T(f_{1},\cdots,f_{j},\cdots,f_{m})-T(f_{1},\cdots,b_{j}f_{j},\cdots,f_{m}).
\end{equation*}

\end{dfn}

For the sake of simplicity and without loss of generality, we only
consider the case of $m=2$ in this paper. Let us state the main
results as follows.

\begin{thm}\label{main-thm}
 Suppose that $\mu$ is a Radon measure with $||\mu||=\infty$. Let
$[b_{1},b_{2},T]$ be defined by (\ref{1.12}). Let $1<p_{1},p_{2},q<+\infty$,
$f_{1}\in L^{p_{1}}(\mu)$, $f_{2}\in L^{p_{2}}(\mu)$, $b_{1}\in
\text{RBMO}(\mu)$ and $b_{2}\in \text{RBMO}(\mu)$. If $T$ is bounded from
$L^{1}(\mu)\times L^{1}(\mu)$ to $ L^{1/2,\infty}(\mu)$, then there
exists a constant $C>0$ such that
\begin{equation}
||[b_{1},b_{2},T](f_{1},f_{2})||_{L^{q}(\mu)}\leq C
||f_{1}||_{L^{p_{1}}(\mu)}||f_{2}||_{L^{p_{2}}(\mu)},
\end{equation}
where $\frac{1}{q}=\frac{1}{p_{1}}+\frac{1}{p_{2}}$.
\end{thm}

\begin{thm}\label{main-thm2}
 Suppose that $\mu$ is a Radon measure with $||\mu||=\infty$. Let
$T_{\sum \vec{b}}$ be defined by (\ref{1.51}) with $m=2$. Let $1<p_{1},p_{2},q<+\infty$,
$f_{1}\in L^{p_{1}}(\mu)$, $f_{2}\in L^{p_{2}}(\mu)$, $b_{1}\in
RBMO(\mu)$ and $b_{2}\in RBMO(\mu)$. If $T$ is bounded from
$L^{1}(\mu)\times L^{1}(\mu)$ to $ L^{1/2,\infty}(\mu)$, then there
exists a constant $C>0$ such that
\begin{equation}
||T_{\sum \vec{b}}(f_{1},f_{2})||_{L^{q}(\mu)}\leq C
||f_{1}||_{L^{p_{1}}(\mu)}||f_{2}||_{L^{p_{2}}(\mu)},
\end{equation}
where $\frac{1}{q}=\frac{1}{p_{1}}+\frac{1}{p_{2}}$.
\end{thm}

\begin{rem} For $||\mu||<\infty$, by Lemma \ref{lem2.1} in Section 2, Theorem \ref{main-thm} also holds if one assumes that $\int_{X}G(f_1,f_2)(x)d\mu(x)=0$ with the operator $G$ be replaced by $T$, $[b_1,T]$,
$[b_2,T]$ and $[b_1,b_2,T]$. Also Theorem \ref{main-thm2} holds if one assumes that $\int_{X}G(f_1,f_2)(x)d\mu(x)=0$ with the operator $G$ be replaced by $T$, $[b_1,T]$ and
$[b_2,T]$.
\end{rem}

\begin{rem} Since the proof of Theorem \ref{main-thm2} is similar with that of  Theorem  \ref{main-thm}, we only give the proof of Theorem \ref{main-thm} in this paper.
\end{rem}

Throughout the paper, $\chi_{E}$
denotes the characteristic function of set $E$.  And $p'$ is the conjugate
index of $p$, namely, $\frac{1}{p}+\frac{1}{p'}=1$. $C$ denotes a positive constant independent of the main parameters involved, but
it may be different in different places.

\section{Proof of Main Result}

To prove Theorem  \ref{main-thm}, we firstly give some notations and
lemmas.

Let $f \in L_{loc}^{1}(\mu)$, the sharp maximal operator is defined
by
 \begin{equation*}
M^{\sharp}f(x)=\sup _{B\ni
x}\frac{1}{\mu(6B)}\int_{B}|f(y)-m_{\widetilde{B}}(f)|d\mu(y)
+\sup_{(B,Q)\in\Delta_{x}}\frac{|m_{B}(f)-m_{Q}(f)|}{K_{B,Q}},
\end{equation*}
where $\Delta_{x}:=\{(B,Q):x\in B\subset Q\ \text{and}\ B,\ Q\ \text{are
doubling balls}\}$ and the non-centered doubling maximal operator is
defined by
$$Nf(x)=\sup_{B\ni x,\atop B \
\text{doubling}}\frac{1}{\mu(B)}\int_{B}|f(y)|d\mu(y).$$ For any $0<\delta
< 1$, we also define that
 \begin{equation*}
M^{\sharp}_{\delta}f(x)=\{M^{\sharp}(|f|^{\delta})(x)\}^{1/\delta}
\end{equation*}
and
\begin{equation*}
N_{\delta}f(x)=\{N(|f|^{\delta})(x)\}^{1/\delta}.
\end{equation*}
We can obtain that for any $f\in L^{1}_{loc}(\mu)$,
\begin{equation*}
|f(x)|\leq N_{\delta}f(x)
\end{equation*} for $\mu-a.e. \ x \in X$.

Let $\rho>1$, $p\in (1,\infty)$ and $r\in (1,p)$, the non-centered
maximal operator $M_{r,(\rho)}f$ is defined by
\begin{equation*}
 M_{r,(\rho)}f(x)=\sup_{B\ni x}\biggl\{\frac{1}{\mu(\rho
B)}\int_{B}|f(y)|^{r}d\mu(y)\biggr\}^{1/r}.
\end{equation*}
When $r=1$, we simply write $ M_{1,(\rho)}f(x)$ as $M_{(\rho)}f$. If
$\rho\geq 5$, then the operator $M_{(\rho)}f$ is bounded on
$L^{p}(\mu)$ for $p>1$ and $M_{r,(\rho)}$ is bounded on $L^{p}(\mu)$
for $p>r$ (see \cite{BD1}).

\begin{lem}\cite{X2}\label{lem2.1}
Let $f \in L^{1}_{loc}(\mu)$ with $\int_{X}f(x)d\mu(x)=0$ if
$||\mu||<\infty$. For $1<p<+\infty$ and $0<\delta<1$, if
$\inf(1,N_{\delta}f)\in L^{p}(\mu)$, then there exists a constant
$C>0$ such that
\begin{equation*}
||N_{\delta}(f)||_{L^{p}(\mu)}\leq
C||M_{\delta}^{\sharp}(f)||_{L^{p}(\mu)}.
\end{equation*}
\end{lem}

\begin{lem}\cite{FYY1}\label{lem2.2}
Let $1\leq p<+\infty$ and $1<\rho <+\infty$. Then $b\in RBMO(\mu)$ if
and only if for any ball $B\subset X$,
\begin{equation*}
\biggl\{\frac{1}{\mu(\rho
B)}\int_{B}|b_{B}-m_{\widetilde{B}}(b)|^{p}d\mu(x)\biggr\}^{1/p}\leq
C,
\end{equation*} and for any two doubling balls
$B\subset Q$,
\begin{equation*}
|m_{B}(b)-m_{Q}(b)|\leq C K_{B,Q}.
\end{equation*}
\end{lem}

\begin{lem}\cite{FYY1}\label{lem2.3}
\begin{equation*}
|m_{\widetilde{6^{j}\frac{6}{5}B}}(b)-m_{\widetilde{B}}(b)|\leq
Cj||b||_{\ast}.
\end{equation*}
\end{lem}

\begin{lem}\cite{FYY1}\label{lem2.10}
$K(B,Q)$ has the following properties:

(i) For all balls $B\subset R\subset Q$, $K(B, R)\leq 2K(B, Q)$.

(ii) For any $\rho\in[1,+\infty)$, there exists a positive constant $C(\rho)$, depending on $\rho$, such
that, for all balls $B\subset Q$ with $r_Q \leq \rho r_B$, $K(B,Q) \leq C(\rho)$.

(iii) There exists a positive constant $C$ such that, for all balls $B$, $K(B, \tilde{B}) \leq C$.

(iv) There exists a positive constant $C$ such that, for all balls $B\subset R\subset Q$, $K(B,Q) \leq K(B, R) + CK(R, Q)$.

(v) There exists a positive constant $C$ such that, for all balls $B\subset R\subset Q$, $K(R,Q)\leq CK(B,Q)$.
\end{lem}

\begin{lem}\cite{WX}\label{lem2.4}
Suppose that $\mu$ is a Radon measure with $||\mu||=\infty$. Let
$T$ be defined by (\ref{1.7}). Let $1<p_{1},p_{2},q<+\infty$, $f_{1}\in
L^{p_{1}}(\mu)$ and $f_{2}\in L^{p_{2}}(\mu)$. If $T$ is bounded
from $L^{1}(\mu)\times L^{1}(\mu)$ to $ L^{1/2,\infty}(\mu)$, then
there exists a constant $C>0$ such that
\begin{equation*}
||T(f_{1},f_{2})||_{L^{q}(\mu)}\leq C
||f_{1}||_{L^{p_{1}}(\mu)}||f_{2}||_{L^{p_{2}}(\mu)},
\end{equation*}
where $\frac{1}{q}=\frac{1}{p_{1}}+\frac{1}{p_{2}}$.
\end{lem}

\begin{lem}\label{lem2.5}
Suppose that $[b_{1},b_{2},T]$ is defined by (\ref{1.12}), $0<\delta<1/2$,
$1<p_{1},p_{2},q<+\infty$, $1<r<q$ and $b_{1},b_{2}\in RBMO(\mu)$. If $T$ is bounded from $L^{1}(\mu)\times L^{1}(\mu)$ to $
L^{1/2,\infty}(\mu)$, then there exists a constant $C>0$ such that
for any $x\in X$, any ball $B$ with $B\ni x$, $f_{1}\in L^{p_{1}}(\mu)$ and $f_{2}\in
L^{p_{2}}(\mu)$, we have the following two cases:

(i)When $l=l(B):=\sup\limits_{x,y\in B}d(x,y)\geq 1$, assume that $\widetilde{B}$ is the smallest $(\alpha,\beta)$-doubling ball of
the form $6^{k}B$ with $k\in {\mathbf{N}}\bigcup\{0\}$, it follows
\begin{equation}\label{2.11}
\begin{split}
&M_{\delta}^{\sharp}[b_{1},b_{2},T](f_{1},f_{2})(x)\leq
C||b_{1}||_{\ast}||b_{2}||_{\ast}M_{r,(6)}(T(f_{1},f_{2}))(x)\\
&\ \ +C||b_{1}||_{\ast}M_{r,(6)}([b_{2},T](f_{1},f_{2}))(x)
+C||b_{2}||_{\ast}M_{r,(6)}([b_{1},T](f_{1},f_{2}))(x)\\
&\ \ +C||b_{1}||_{\ast}||b_{2}||_{\ast}M_{p_{1},(5)}f_{1}(x)M_{p_{2},(5)}f_{2}(x)
+C||b_{1}||_{\ast}||b_{2}||_{\ast}M_{r,(6)}(T(f_{1}\chi_{\frac{6}{5}\tilde{B}},f_{2}\chi_{\frac{6}{5}\tilde{B}}))(x)\\
&\ \ +C||b_{1}||_{\ast}M_{r,(6)}([b_{2},T](f_{1}\chi_{\frac{6}{5}\tilde{B}},f_{2}\chi_{\frac{6}{5}\tilde{B}}))(x)
+C||b_{2}||_{\ast}M_{r,(6)}([b_{1},T](f_{1}\chi_{\frac{6}{5}\tilde{B}},f_{2}\chi_{\frac{6}{5}\tilde{B}}))(x),\\
\end{split}
\end{equation}

\begin{equation}\label{2.12}
\begin{split}
M_{\delta}^{\sharp}[b_{1},T](f_{1},f_{2})(x)&\leq
C||b_{1}||_{\ast}M_{r,(6)}(T(f_{1},f_{2}))(x)
+C||b_{1}||_{\ast}M_{p_{1},(5)}f_{1}(x)M_{p_{2},(5)}f_{2}(x)\\
&\ \ +C||b_{1}||_{\ast}M_{r,(6)}(T(f_{1}\chi_{\frac{6}{5}\tilde{B}},f_{2}\chi_{\frac{6}{5}\tilde{B}}))(x)
\end{split}
\end{equation}
and
\begin{equation}\label{2.13}
\begin{split}
M_{\delta}^{\sharp}[b_{2},T](f_{1},f_{2})(x)&\leq
C||b_{2}||_{\ast}M_{r,(6)}(T(f_{1},f_{2}))(x)
+C||b_{2}||_{\ast}M_{p_{1},(5)}f_{1}(x)M_{p_{2},(5)}f_{2}(x)\\
&\ \ +C||b_{2}||_{\ast}M_{r,(6)}(T(f_{1}\chi_{\frac{6}{5}\tilde{B}},f_{2}\chi_{\frac{6}{5}\tilde{B}}))(x).
\end{split}
\end{equation}

(ii)When $0<l<1$, assume that $B_{0}$, $Q_{0}$ and $\tilde{B}_{0}$ are concentric with $B$, $Q$ and $\tilde{B}$ respectively and $l(B_{0})=l(B)^{\alpha}$, $l(Q_{0})=l(Q)^{\alpha}$, $l(\tilde{B}_{0})=l(\tilde{B})^{\alpha}$, it follows
\begin{equation}\label{2.14}
\begin{split}
&M_{\delta}^{\sharp}[b_{1},b_{2},T](f_{1},f_{2})(x)\leq
C||b_{1}||_{\ast}||b_{2}||_{\ast}M_{r,(6)}(T(f_{1},f_{2}))(x)\\
&\ \ +C||b_{1}||_{\ast}M_{r,(6)}([b_{2},T](f_{1},f_{2}))(x)
+C||b_{2}||_{\ast}M_{r,(6)}([b_{1},T](f_{1},f_{2}))(x)\\
&\ \ +C||b_{1}||_{\ast}||b_{2}||_{\ast}M_{p_{1},(5)}f_{1}(x)M_{p_{2},(5)}f_{2}(x)
+C||b_{1}||_{\ast}||b_{2}||_{\ast}M_{r,(6)}(T(f_{1}\chi_{\frac{6}{5}\tilde{B}_{0}},f_{2}\chi_{\frac{6}{5}\tilde{B}_{0}}))(x)\\
&\ \ +C||b_{1}||_{\ast}M_{r,(6)}([b_{2},T](f_{1}\chi_{\frac{6}{5}\tilde{B}_{0}},f_{2}\chi_{\frac{6}{5}\tilde{B}_{0}}))(x)
+C||b_{2}||_{\ast}M_{r,(6)}([b_{1},T](f_{1}\chi_{\frac{6}{5}\tilde{B}_{0}},f_{2}\chi_{\frac{6}{5}\tilde{B}_{0}}))(x),\\
\end{split}
\end{equation}

\begin{equation}\label{2.15}
\begin{split}
M_{\delta}^{\sharp}[b_{1},T](f_{1},f_{2})(x)&\leq
C||b_{1}||_{\ast}M_{r,(6)}(T(f_{1},f_{2}))(x)
+C||b_{1}||_{\ast}M_{p_{1},(5)}f_{1}(x)M_{p_{2},(5)}f_{2}(x)\\
&\ \ +C||b_{1}||_{\ast}M_{r,(6)}(T(f_{1}\chi_{\frac{6}{5}\tilde{B}_{0}},f_{2}\chi_{\frac{6}{5}\tilde{B}_{0}}))(x)
\end{split}
\end{equation}
and
\begin{equation}\label{2.16}
\begin{split}
M_{\delta}^{\sharp}[b_{2},T](f_{1},f_{2})(x)&\leq
C||b_{2}||_{\ast}M_{r,(6)}(T(f_{1},f_{2}))(x)
+C||b_{2}||_{\ast}M_{p_{1},(5)}f_{1}(x)M_{p_{2},(5)}f_{2}(x)\\
&\ \ +C||b_{2}||_{\ast}M_{r,(6)}(T(f_{1}\chi_{\frac{6}{5}\tilde{B}_{0}},f_{2}\chi_{\frac{6}{5}\tilde{B}_{0}}))(x).
\end{split}
\end{equation}

\end{lem}

\begin{proof}
Because $L^{\infty}(\mu)$ with bounded support is dense in
$L^{p}(\mu)$ for $1<p<+\infty$, we only consider the situation of
$f_{1},f_{2}\in L^{\infty}(\mu)$ with bounded support. Also, by
Corollary 3.11 in \cite{T2}, without loss of generality, we can assume
that $b_{1},b_{2}\in L^{\infty}(\mu)$. Next we divide two cases for proving the result.

{\bf Case 1: $l(B)=l\geq 1$}. As in the proof of Theorem 9.1 in \cite{T2}, to obtain (\ref{2.11}), it suffices to show that
\begin{equation}\label{2.17}
\begin{split}
&\biggl(\frac{1}{\mu(6B)}\int_{B}||[b_{1},b_{2},T](f_{1},f_{2})(z)|^{\delta}-|h_{B}|^{\delta}|d\mu(z)\biggr
)^{1/\delta}\\
\leq &C||b_{1}||_{\ast}||b_{2}||_{\ast}M_{r,(6)}(T(f_{1},f_{2}))(x)
+C||b_{1}||_{\ast}M_{r,(6)}([b_{2},T](f_{1},f_{2}))(x)\\
&+C||b_{2}||_{\ast}M_{r,(6)}([b_{1},T](f_{1},f_{2}))(x)+C||b_{1}||_{\ast}||b_{2}||_{\ast}M_{p_{1},(5)}f_{1}(x)M_{p_{2},(5)}f_{2}(x),
\end{split}
\end{equation}
 holds for
any $x$ and ball $B$ with $x\in B$, and
\begin{equation}\label{2.18}
\begin{split}
|h_{B}-h_{Q}|\leq&
CK_{B,Q}^{4}\biggr[||b_{1}||_{\ast}||b_{2}||_{\ast}M_{r,(6)}(T(f_{1},f_{2}))(x)
+||b_{1}||_{\ast}M_{r,(6)}([b_{2},T](f_{1},f_{2}))(x)\\
&+||b_{2}||_{\ast}M_{r,(6)}([b_{1},T](f_{1},f_{2}))(x)
+||b_{1}||_{\ast}||b_{2}||_{\ast}M_{p_{1},(5)}f_{1}(x)M_{p_{2},(5)}f_{2}(x)\\
&+||b_{1}||_{\ast}||b_{2}||_{\ast}M_{r,(6)}(T(f_{1}\chi_{\frac{6}{5}\tilde{B}},f_{2}\chi_{\frac{6}{5}\tilde{B}}))(x)
+||b_{1}||_{\ast}M_{r,(6)}([b_{2},T](f_{1}\chi_{\frac{6}{5}\tilde{B}},f_{2}\chi_{\frac{6}{5}\tilde{B}}))(x)\\
&+||b_{2}||_{\ast}M_{r,(6)}([b_{1},T](f_{1}\chi_{\frac{6}{5}\tilde{B}},f_{2}\chi_{\frac{6}{5}\tilde{B}}))(x)\biggr]
\end{split}
\end{equation}
for all balls $B\subset Q$ with $x\in B$, where $B$ is an arbitrary
ball, $Q$ is a doubling ball, $\widetilde{B}$ is the smallest $(\alpha,\beta)$-doubling ball of
the form $6^{k}B$ with $k\in {\mathbf{N}}\bigcup\{0\}$. We denote
\begin{equation*}
\begin{split}
h_{B}:=& m_{B}\biggl[T\biggl((b_{1}-
m_{\tilde{B}}(b_{1}))f_{1}\chi_{\frac{6}{5}B},(b_{2}-m_{\tilde{B}}(b_{2}))f_{2}\chi_{X\backslash\frac{6}{5}B}\biggr)\\
&\ \ +T\biggl((b_{1}-
m_{\tilde{B}}(b_{1}))f_{1}\chi_{X\backslash\frac{6}{5}B},(b_{2}-m_{\tilde{B}}(b_{2}))f_{2}\chi_{\frac{6}{5}B}\biggr)\\
&\ \ +T\biggl((b_{1}-
m_{\tilde{B}}(b_{1}))f_{1}\chi_{X\backslash\frac{6}{5}B},(b_{2}-m_{\tilde{B}}(b_{2}))f_{2}\chi_{X\backslash\frac{6}{5}B}\biggr)\biggr]
\end{split}
\end{equation*}
and
\begin{equation*}
\begin{split}
h_{Q}:=& m_{Q}\biggl[T\biggl((b_{1}-
m_{Q}(b_{1}))f_{1}\chi_{\frac{6}{5}Q},(b_{2}-m_{Q}(b_{2}))f_{2}\chi_{X\backslash\frac{6}{5}Q}\biggr)\\
&\ \ +T\biggl((b_{1}-
m_{Q}(b_{1}))f_{1}\chi_{X\backslash\frac{6}{5}Q},(b_{2}-m_{Q}(b_{2}))f_{2}\chi_{\frac{6}{5}Q}\biggr)\\
&\ \ +T\biggl((b_{1}-
m_{Q}(b_{1}))f_{1}\chi_{X\backslash\frac{6}{5}Q},(b_{2}-m_{Q}(b_{2}))f_{2}\chi_{X\backslash\frac{6}{5}Q}\biggr)\biggr].
\end{split}
\end{equation*}
It is easy to see that
$$[b_{1},b_{2},T]=T((b_{1}-b_{1}(z))f_{1},(b_{2}-b_{2}(z))f_{2})$$
and
\begin{equation*}
\begin{split}
&T((b_{1}-m_{\tilde{B}}(b_{1}))f_{1},(b_{2}-m_{\tilde{B}}(b_{2}))f_{2})\\
=&T((b_{1}-b_{1}(z)+b_{1}(z)-m_{\tilde{B}}(b_{1}))f_{1},(b_{2}-b_{2}(z)+b_{2}(z)-m_{\tilde{B}}(b_{2}))f_{2})\\
=&(b_{1}(z)-m_{\tilde{B}}(b_{1}))(b_{2}(z)-m_{\tilde{B}}(b_{2}))T(f_{1},f_{2})\\
&\   +(b_{1}(z)-m_{\tilde{B}}(b_{1}))T(f_{1},(b_{2}-b_{2}(z))f_{2})\\
&\
+(b_{2}(z)-m_{\tilde{B}}(b_{2}))T((b_{1}-b_{1}(z))f_{1},f_{2})+T((b_{1}-b_{1}(z))f_{1},(b_{2}-b_{2}(z))f_{2}).
\end{split}
\end{equation*}
Then
\begin{equation*}
\begin{split}
&\biggl(\frac{1}{\mu(6B)}\int_{B}||[b_{1},b_{2},T](f_{1},f_{2})(z)|^{\delta}-|h_{B}|^{\delta}|d\mu(z)\biggr
)^{1/\delta}\\
\leq&
C\biggl(\frac{1}{\mu(6B)}\int_{B}|[b_{1},b_{2},T](f_{1},f_{2})(z)-h_{B}|^{\delta}d\mu(z)\biggr)^{1/\delta}\\
\leq&
C\biggl(\frac{1}{\mu(6B)}\int_{B}|(b_{1}(z)-m_{\tilde{B}}(b_{1}))(b_{2}(z)-m_{\tilde{B}}(b_{2}))T(f_{1},f_{2})(z)|^{\delta}d\mu(z)\biggr)^{1/\delta}\\
&\ \ +C\biggl(\frac{1}{\mu(6B)}\int_{B}|(b_{1}(z)-m_{\tilde{B}}(b_{1}))T(f_{1},(b_{2}-b_{2}(z))f_{2})(z)|^{\delta}d\mu(z)\biggr)^{1/\delta}\\
&\ \ +C\biggl(\frac{1}{\mu(6B)}\int_{B}|(b_{2}(z)-m_{\tilde{B}}(b_{2}))T((b_{1}-b_{1}(z))f_{1},f_{2})(z)|^{\delta}d\mu(z)\biggr)^{1/\delta}\\
&\ \
+C\biggl(\frac{1}{\mu(6B)}\int_{B}|T((b_{1}-m_{\tilde{B}}(b_{1}))f_{1},(b_{2}-m_{\tilde{B}}(b_{2}))f_{2})(z)-h_{B}|^{\delta}d\mu(z)\biggr)^{1/\delta}\\
=&:E_{1}+E_{2}+E_{3}+E_{4}.
\end{split}
\end{equation*}

We first estimate $E_{1}$. Choosing $r_{1},r_{2}>1$ such that
$\dfrac{1}{r}+\dfrac{1}{r_{1}}+\dfrac{1}{r_{2}}=\dfrac{1}{\delta}$.
By H\"{o}lder's inequality and Lemma \ref{lem2.2}, it follows

\begin{equation}\label{2.81}
\begin{split}
E_{1}&\leq
C\biggl(\frac{1}{\mu(6B)}\int_{B}|b_{1}(z)-m_{\tilde{B}}(b_{1})|^{r_{1}}d\mu(z)\biggr)^{1/r_{1}}\\
&\ \ \times\biggl(\frac{1}{\mu(6B)}\int_{B}|b_{2}(z)-m_{\tilde{B}}(b_{2})|^{r_{2}}d\mu(z)\biggr)^{1/r_{2}}\\
&\ \ \times\biggl(\frac{1}{\mu(6B)}\int_{B}|T(f_{1},f_{2})|^{r}d\mu(z)\biggr)^{1/r}\\
&\leq C||b_{1}||_{\ast}||b_{2}||_{\ast}M_{r,(6)}(T(f_{1},f_{2}))(x).
\end{split}
\end{equation}

For $E_{2}$, choosing $s>1$ such that
$\dfrac{1}{s}+\dfrac{1}{r}=\dfrac{1}{\delta}$, by H\"{o}lder's
inequality and Lemma \ref{lem2.2}, we obtain

\begin{equation}\label{2.91}
\begin{split}
E_{2}&\leq
C\biggl(\frac{1}{\mu(6B)}\int_{B}|b_{1}(z)-m_{\tilde{B}}(b_{1})|^{s}d\mu(z)\biggr)^{1/s}\\
&\ \ \times\biggl(\frac{1}{\mu(6B)}\int_{B}|[b_{2},T](f_{1},f_{2})|^{r}d\mu(z)\biggr)^{1/r}\\
&\leq C||b_{1}||_{\ast}M_{r,(6)}([b_{2},T](f_{1},f_{2}))(x).
\end{split}
\end{equation}

Similar to estimate $E_{2}$, we immediately have
\begin{equation}\label{2.92}
E_{3}\leq C||b_{2}||_{\ast}M_{r,(6)}([b_{1},T](f_{1},f_{2}))(x).
\end{equation}

Let us turn to estimate $E_{4}$. Denote
$f_{j}^{1}=f_{j}\chi_{\frac{6}{5}B}$ and $f_{j}^{2}=f_{j}-f_{j}^{1}$
for $j=1,2$, we have
\begin{equation*}
\begin{split}
&|T((b_{1}-m_{\tilde{B}}(b_{1}))f_{1},(b_{2}-m_{\tilde{B}}(b_{2}))f_{2})(z)-h_{B}|\\
\leq&|T((b_{1}-m_{\tilde{B}}(b_{1}))f^{1}_{1},(b_{2}-m_{\tilde{B}}(b_{2}))f^{1}_{2})(z)|\\
&\ \ +|T((b_{1}-m_{\tilde{B}}(b_{1}))f^{1}_{1},(b_{2}-m_{\tilde{B}}(b_{2}))f^{2}_{2})(z)
-m_{B}[T((b_{1}-m_{\tilde{B}}(b_{1}))f^{1}_{1},(b_{2}-m_{\tilde{B}}(b_{2}))f^{2}_{2})]|\\
&\ \ +|T((b_{1}-m_{\tilde{B}}(b_{1}))f^{2}_{1},(b_{2}-m_{\tilde{B}}(b_{2}))f^{1}_{2})(z)
-m_{B}[T((b_{1}-m_{\tilde{B}}(b_{1}))f^{2}_{1},(b_{2}-m_{\tilde{B}}(b_{2}))f^{1}_{2})]|\\
&\ \ +|T((b_{1}-m_{\tilde{B}}(b_{1}))f^{2}_{1},(b_{2}-m_{\tilde{B}}(b_{2}))f^{2}_{2})(z)
-m_{B}[T((b_{1}-m_{\tilde{B}}(b_{1}))f^{2}_{1},(b_{2}-m_{\tilde{B}}(b_{2}))f^{2}_{2})]|\\
=&:E_{41}(z)+E_{42}(z)+E_{43}(z)+E_{44}(z).
\end{split}
\end{equation*}
Then
\begin{equation*}
\begin{split}
E_{4}&\leq
C\biggl(\frac{1}{\mu(6B)}\int_{B}E_{41}(z)^{\delta}d\mu(z)\biggr)^{1/\delta}
+C\biggl(\frac{1}{\mu(6B)}\int_{B}E_{42}(z)^{\delta}d\mu(z)\biggr)^{1/\delta}\\
&\ \ +
C\biggl(\frac{1}{\mu(6B)}\int_{B}E_{43}(z)^{\delta}d\mu(z)\biggr)^{1/\delta}
+
C\biggl(\frac{1}{\mu(6B)}\int_{B}E_{44}(z)^{\delta}d\mu(z)\biggr)^{1/\delta}\\
&=:E_{41}+E_{42}+E_{43}+E_{44}.
\end{split}
\end{equation*}

To estimate $E_{41}$, we need the classical Kolmogorov's theorem:
Let $(X,\mu)$ be a probability measure space and let $0<p<q<\infty$,
then there exists a constant $C>0$, such that $||f||_{L^{p}(\mu)}\le
C||f||_{L^{q,\infty}(\mu)}$ for any measurable function $f$. Let
$p=\delta$ and $q=1/2$ such that $0<\delta<1/2$. Using Kolmogorov's
theorem, Lemma \ref{lem2.2} and H\"{o}lder's inequality, we obtain
\begin{equation*}
\begin{split}
E_{41}\leq&
C||T((b_{1}-m_{\tilde{B}}(b_{1}))f_{1}^{1},(b_{2}-m_{\tilde{B}}(b_{2}))f_{2}^{1})||_{L^{1/2,\infty}(\frac{6}{5}B,\frac{d\mu(z)}{\mu(6B)})}\\
\leq&
C\frac{1}{\mu(6B)}\int_{\frac{6}{5}B}|(b_{1}(z)-m_{\tilde{B}}(b_{1}))f_{1}(z)|d\mu(z)\\
&\ \ \times\frac{1}{\mu(6B)}\int_{\frac{6}{5}B}|(b_{2}(z)-m_{\tilde{B}}(b_{2}))f_{2}(z)|d\mu(z)\\
\leq&
C(\frac{1}{\mu(6B)}\int_{\frac{6}{5}B}|b_{1}(z)-m_{\tilde{B}}(b_{1})|^{p'_{1}}d\mu(z))^{1/p'_{1}}\\
&\ \ \times(\frac{1}{\mu(6B)}\int_{\frac{6}{5}B}|f_{1}(z)|^{p_{1}}d\mu(z))^{1/p_{1}}\\
&\ \ \times(\frac{1}{\mu(6B)}\int_{\frac{6}{5}B}|b_{2}(z)-m_{\tilde{B}}(b_{2})|^{p'_{2}}d\mu(z))^{1/p'_{2}}\\
&\ \ \times(\frac{1}{\mu(6B)}\int_{\frac{6}{5}B}|f_{2}(z)|^{p_{2}}d\mu(z))^{1/p_{2}}\\
\leq&
C||b_{1}||_{\ast}||b_{2}||_{\ast}M_{p_{1},(5)}f_{1}(x)M_{p_{2},(5)}f_{2}(x).
\end{split}
\end{equation*}

To compute $E_{42}$, let $z,y\in B$, $y_{1}\in \frac{6}{5}B$ and $y_{2}\in X\backslash \frac{6}{5}B$, then $\max\limits_{1\leq i\leq
2}d(z,y_{i})\geq d(z,y_{2})\geq Cl(B)\geq Cl(B)^{\alpha}\geq Cd(z,y)^{\alpha}$.  By Definition \ref{def1.6}, Lemma \ref{lem2.2}, Lemma \ref{lem2.3},
H\"{o}lder's inequality and the properties of $\lambda$, it follows

\begin{equation*}
\begin{split}
&|T((b_{1}-m_{\tilde{B}}(b_{1}))f^{1}_{1},(b_{2}-m_{\tilde{B}}(b_{2}))f^{2}_{2})(z)
-T((b_{1}-m_{\tilde{B}}(b_{1}))f^{1}_{1},(b_{2}-m_{\tilde{B}}(b_{2}))f^{2}_{2})(y)|\\
\leq&C\int_{X\backslash\frac{6}{5}B}\int_{\frac{6}{5}B}|K(z,y_1,y_2)-K(y,y_1,y_2)| |\prod_{i=1}^{2}(b_{i}(y_{i})-m_{\tilde{B}}(b_i))f_{i}(y_i)|d\mu(y_1)d\mu(y_2)\\
\leq &C\int_{X\backslash\frac{6}{5}B}\int_{\frac{6}{5}B}\frac{d(z,y)^{\delta}}
{\biggl[\sum\limits_{i=1}^{2}d(z,y_{i})\biggr]^{\delta/\alpha}\biggl[\sum\limits_{i=1}^{2}\lambda(z,d(z,y_{i}))\biggr]^{2}} \\ &\ \ \ \times|\prod_{i=1}^{2}(b_{i}(y_{i})-m_{\tilde{B}}(b_i))f_{i}(y_i)|d\mu(y_1)d\mu(y_2)\\
\leq& C\int_{\frac{6}{5}B}
\frac{|b_{1}(y_{1})-m_{\tilde{B}}(b_{1})||f_{1}(y_{1})|}{\lambda(z,d(z,y_{1}))}d\mu(y_{1})\\
&\ \ \ \times\int_{X\backslash\frac{6}{5}B}
\frac{d(z,y)^{\delta}}
{d(z,y_{2})^{\delta/\alpha}}\frac{|b_{2}(y_{2})-m_{\tilde{B}}(b_{2})||f_{2}(y_{2})|}{\lambda(z,d(z,y_{2}))}d\mu(y_{2})\\
\leq& C\frac{\mu(6B)}{\lambda(x_{B},\frac{6}{5}r_{B})}\frac{1}{\mu(6B)}\int_{\frac{6}{5}B}
|b_{1}(y_{1})-m_{\tilde{B}}(b_{1})||f_{1}(y_{1})|d\mu(y_{1})\\
&\ \ \ \times\sum_{k=1}^{\infty}\int_{6^{k}\frac{6}{5}B\backslash6^{k-1}\frac{6}{5}B}
\frac{d(z,y)^{\delta}}
{d(z,y_{2})^{\delta/\alpha}}\frac{|b_{2}(y_{2})-m_{\tilde{B}}(b_{2})||f_{2}(y_{2})|}{\lambda(z,d(z,y_{2}))}d\mu(y_{2})\\
\leq&
C(\frac{1}{\mu(6B)}\int_{\frac{6}{5}B}|b_{1}(y_{1})-m_{\tilde{B}}(b_{1})|^{p'_{1}}d\mu(y_{1}))^{1/p'_{1}}\\
&\ \ \ \times(\frac{1}{\mu(6B)}\int_{\frac{6}{5}B}|f_{1}(y_{1})|^{p_{1}}d\mu(y_{1}))^{1/p_{1}}\\
&\ \ \ \times
\sum_{k=1}^{\infty}\frac{1}{\lambda(x_{B},6^{k-1}\frac{6}{5}r_{B})}6^{-k\delta/\alpha}l^{\delta(1-1/\alpha)}
\int_{6^{k}\frac{6}{5}B}|b_{2}(y_{2})-m_{\tilde{B}}
(b_{2})||f_{2}(y_{2})|d\mu(y_{2})\\
\leq&
C||b_{1}||_{\ast}M_{p_{1},(5)}f_{1}(x)\sum_{k=1}^{\infty}6^{-k\delta/\alpha}
\frac{\mu(5\times6^{k}\frac{6}{5}B)}{\lambda(x_{B},6^{k-1}\frac{6}{5}r_{B})}
\frac{1}{\mu(5\times6^{k}\frac{6}{5}B)}\\
&\ \ \ \times\int_{6^{k}\frac{6}{5}B}
|b_{2}(y_{2})-m_{\widetilde{6^{k}\frac{6}{5}B}}(b_{2})+m_{\widetilde{6^{k}\frac{6}{5}B}}(b_{2})-m_{\tilde{B}}(b_{2})||f_{2}(y_{2})|d\mu(y_{2})\\
\leq&
C||b_{1}||_{\ast}M_{p_{1},(5)}f_{1}(x)\sum_{k=1}^{\infty}6^{-k\delta/\alpha}\\
&\ \ \ \times\biggl[\biggl(\frac{1}{\mu(5\times6^{k}\frac{6}{5}B)}
\int_{6^{k}\frac{6}{5}B}|b_{2}(y_{2})-m_{\widetilde{6^{k}\frac{6}{5}B}}(b_{2})|^{p'_{2}}d\mu(y_{2})\biggr)^{1/p'_{2}}\\
&\ \ \ \times
\biggl(\frac{1}{\mu(5\times6^{k}\frac{6}{5}B)}\int_{6^{k}\frac{6}{5}B}|f_{2}(y_{2})|^{p_{2}}d\mu(y_{2})\biggr)^{1/p_{2}}\\
&\ \ \ +C
k||b_{2}||_{\ast}\frac{1}{\mu(5\times6^{k}\frac{6}{5}B)}\int_{6^{k}\frac{6}{5}B}|f_{2}(y_{2})|d\mu(y_{2})\biggr]\\
\leq&
C||b_{1}||_{\ast}||b_{2}||_{\ast}M_{p_{1},(5)}f_{1}(x)M_{p_{2},(5)}f_{2}(x).\\
\end{split}
\end{equation*}
Taking the mean over $y\in B$, we have
\begin{equation*}
E_{42}(z)\leq
C||b_{1}||_{\ast}||b_{2}||_{\ast}M_{p_{1},(5)}f_{1}(x)M_{p_{2},(5)}f_{2}(x).
\end{equation*}
Therefore
\begin{equation*}
E_{42}\leq
C||b_{1}||_{\ast}||b_{2}||_{\ast}M_{p_{1},(5)}f_{1}(x)M_{p_{2},(5)}f_{2}(x).
\end{equation*}

Similarly, we get
\begin{equation*}
E_{43}\leq
C||b_{1}||_{\ast}||b_{2}||_{\ast}M_{p_{1},(5)}f_{1}(x)M_{p_{2},(5)}f_{2}(x).
\end{equation*}

For $E_{44}$, by Definition \ref{def1.6}, Lemma \ref{lem2.2},
H\"{o}lder's inequality and the properties of $\lambda$, we obtain
\begin{equation*}
\begin{split}
&|T((b_{1}-m_{\tilde{B}}(b_{1}))f^{2}_{1},(b_{2}-m_{\tilde{B}}(b_{2}))f^{2}_{2})(z)
-T((b_{1}-m_{\tilde{B}}(b_{1}))f^{2}_{1},(b_{2}-m_{\tilde{B}}(b_{2}))f^{2}_{2})(y)|\\
\leq &C\int_{X\backslash
\frac{6}{5}B}\int_{X\backslash\frac{6}{5}B}|K(z,y_{1},y_{2})-K(y,y_{1},y_{2})|\\
&\ \ \ \times|\prod_{i=1}^{2}(b_{i}(y_{i})-m_{\tilde{B}}(b_{i}))f_{i}(y_{i})|d\mu(y_{1})d\mu(y_{2})\\
\leq &C\int_{X\backslash
\frac{6}{5}B}\int_{X\backslash\frac{6}{5}B}\frac{d(z,y)^{\delta}
|\prod_{i=1}^{2}(b_{i}(y_{i})-m_{\tilde{B}}(b_{i})f_{i}(y_{i})|d\mu(y_{1})d\mu(y_{2})}
{(d(z,y_{1})+d(z,y_{2}))^{\delta/\alpha}[\sum\limits_{j=1}^{2}\lambda(z,d(z,y_{j}))]^{2}}\\
\leq &C\prod_{i=1}^{2}\int_{X\backslash
\frac{6}{5}B}\frac{d(z,y)^{\delta_{i}}
|b_{i}(y_{i})-m_{\tilde{B}}(b_{i})||f_{i}(y_{i})|d\mu(y_{i})}{d(z,y_{i})^{\delta_{i}/\alpha}\lambda(z,d(z,y_{i}))}\\
\leq
&C\prod_{i=1}^{2}\sum\limits_{k=1}^{\infty}6^{-k\delta_{i}/\alpha}l^{\delta(1-1/\alpha)}\frac{\mu(5\times6^{k}\frac{6}{5}B)}
{\lambda(x_{B},6^{k-1}\frac{6}{5}r_{B})}\frac{1}{\mu(5\times6^{k}\frac{6}{5}B)}\\
&\ \ \times\int_{6^{k}\frac{6}{5}B}|b_{i}(y_{i})-m_{\tilde{B}}(b_{i})||f_{i}(y_i)|d\mu(y_{i})\\
\leq
&C\prod_{i=1}^{2}\sum_{k=1}^{\infty}6^{-k\delta_{i}/\alpha}(\frac{1}{\mu(5\times6^{k}\frac{6}{5}B)}
\int_{6^{k}\frac{6}{5}B}|b_{i}(y_{i})-m_{\tilde{B}}(b_{i})|^{p'_{i}}d\mu(y_{i}))^{1/p'_{i}}\\
&\ \ \ \times(\frac{1}{\mu(5\times6^{k}\frac{6}{5}B)}
\int_{6^{k}\frac{6}{5}B}|f_{i}(y_i)|^{p_{i}})^{1/p_{i}}\\
\leq
&C\prod_{i=1}^{2}\sum_{k=1}^{\infty}6^{-k\delta_{i}/\alpha}M_{p_{i},(5)}f_{i}(x)(\frac{1}{\mu(5\times6^{k}\frac{6}{5}B)}
\int_{6^{k}\frac{6}{5}B}|b_{i}(y_{i})-m_{\widetilde{6^{k}\frac{6}{5}B}}(b_i)\\
&\ \ \ +m_{\widetilde{6^{k}\frac{6}{5}B}}(b_{i})-m_{\tilde{B}}(b_{i})|^{p'_{i}}d\mu(y_{i}))^{1/p'_{i}}\\
\leq
&C\prod_{i=1}^{2}\sum_{k=1}^{\infty}k6^{-k\delta_{i}/\alpha}||b_{i}||_{\ast}M_{p_{i},(5)}f_{i}(x)\\
\leq&
C||b_{1}||_{\ast}||b_{2}||_{\ast}M_{p_{1},(5)}f_{1}(x)M_{p_{2},(5)}f_{2}(x).
\end{split}
\end{equation*}
where $\delta_{1},\delta_{2}>0$ and $\delta_{1}+\delta_{2}=\delta$.

Taking the mean over $y\in B$, then
\begin{equation*}
E_{44}(z)\leq
C||b_{1}||_{\ast}||b_{2}||_{\ast}M_{p_{1},(5)}f_{1}(x)M_{p_{2},(5)}f_{2}(x).
\end{equation*}
Thus
\begin{equation*}
E_{44}\leq
C||b_{1}||_{\ast}||b_{2}||_{\ast}M_{p_{1},(5)}f_{1}(x)M_{p_{2},(5)}f_{2}(x).
\end{equation*}
So (\ref{2.17}) can be obtained.

Next we prove (\ref{2.18}). Consider two balls $B\subset Q$ with $x\in B$,
where $B$ is an arbitrary ball and $Q$ is a doubling ball. Recall that $\widetilde{B}$ is the smallest $(\alpha,\beta)$-doubling ball of
the form $6^{k}B$ with $k\in {\mathbf{N}}\bigcup\{0\}$. Now we have the two cases: $B\subset Q\subset\tilde{B}$ or $B\subset\tilde{B}\subset Q$.
By Lemma \ref{lem2.10}, we can obtain the following facts. If $B\subset Q\subset\tilde{B}$, we have $K_{Q,\tilde{B}}\leq CK_{B,\tilde{B}}\leq C$. If $B\subset\tilde{B}\subset Q$, we have $K_{\tilde{B},Q}\leq CK_{B,Q}$. Without loss of generality, we only consider $B\subset\tilde{B}\subset Q$ in this paper. The another case can be considered by the same method.

Let
$N=N_{\tilde{B},Q}+1$. It is easy to see that
\begin{equation*}
|h_{B}-h_{Q}|\leq |h_{B}-h_{\tilde{B}}|+|h_{\tilde{B}}-h_{Q}|.
\end{equation*}
Now we first consider $|h_{\tilde{B}}-h_{Q}|$. We write
\begin{equation*}
\begin{split}
|h_{\tilde{B}}-h_{Q}|\leq & \biggl|m_{\tilde{B}}\biggl[T\biggl((b_{1}-
m_{\tilde{B}}(b_{1}))f_{1}\chi_{\frac{6}{5}\tilde{B}},(b_{2}-m_{\tilde{B}}(b_{2}))f_{2}\chi_{X\backslash\frac{6}{5}\tilde{B}}\biggr)\\
&\ \ +T\biggl((b_{1}-
m_{\tilde{B}}(b_{1}))f_{1}\chi_{X\backslash\frac{6}{5}\tilde{B}},(b_{2}-m_{\tilde{B}}(b_{2}))f_{2}\chi_{\frac{6}{5}\tilde{B}}\biggr)\\
&\ \ +T\biggl((b_{1}-
m_{\tilde{B}}(b_{1}))f_{1}\chi_{X\backslash\frac{6}{5}\tilde{B}},(b_{2}-m_{\tilde{B}}(b_{2}))f_{2}\chi_{X\backslash\frac{6}{5}\tilde{B}}\biggr)\biggr]\\
&\ \ -m_{\tilde{B}}\biggl[T\biggl((b_{1}-
m_{Q}(b_{1}))f_{1}\chi_{\frac{6}{5}\tilde{B}},(b_{2}-m_{Q}(b_{2}))f_{2}\chi_{X\backslash\frac{6}{5}\tilde{B}}\biggr)\\
&\ \ +T\biggl((b_{1}-
m_{Q}(b_{1}))f_{1}\chi_{X\backslash\frac{6}{5}\tilde{B}},(b_{2}-m_{Q}(b_{2}))f_{2}\chi_{\frac{6}{5}\tilde{B}}\biggr)\\
&\ \ +T\biggl((b_{1}-
m_{Q}(b_{1}))f_{1}\chi_{X\backslash\frac{6}{5}\tilde{B}},(b_{2}-m_{Q}(b_{2}))f_{2}\chi_{X\backslash\frac{6}{5}\tilde{B}}\biggr)\biggr]\biggr|\\
&\ \ +\biggl|m_{\tilde{B}}\biggl[T\biggl((b_{1}-
m_{Q}(b_{1}))f_{1}\chi_{\frac{6}{5}\tilde{B}},(b_{2}-m_{Q}(b_{2}))f_{2}\chi_{X\backslash\frac{6}{5}\tilde{B}}\biggr)\\
&\ \ +T\biggl((b_{1}-
m_{Q}(b_{1}))f_{1}\chi_{X\backslash\frac{6}{5}\tilde{B}},(b_{2}-m_{Q}(b_{2}))f_{2}\chi_{\frac{6}{5}\tilde{B}}\biggr)\\
&\ \ +T\biggl((b_{1}-
m_{Q}(b_{1}))f_{1}\chi_{X\backslash\frac{6}{5}\tilde{B}},(b_{2}-m_{Q}(b_{2}))f_{2}\chi_{X\backslash\frac{6}{5}\tilde{B}}\biggr)\biggr]\\
&\ \ -m_{Q}\biggl[T\biggl((b_{1}-
m_{Q}(b_{1}))f_{1}\chi_{\frac{6}{5}Q},(b_{2}-m_{Q}(b_{2}))f_{2}\chi_{X\backslash\frac{6}{5}Q}\biggr)\\
&\ \ +T\biggl((b_{1}-
m_{Q}(b_{1}))f_{1}\chi_{X\backslash\frac{6}{5}Q},(b_{2}-m_{Q}(b_{2}))f_{2}\chi_{\frac{6}{5}Q}\biggr)\\
&\ \ +T\biggl((b_{1}-
m_{Q}(b_{1}))f_{1}\chi_{X\backslash\frac{6}{5}Q},(b_{2}-m_{Q}(b_{2}))f_{2}\chi_{X\backslash\frac{6}{5}Q}\biggr)\biggr]\biggr|\\
=:&F_{1}+F_{2}.
\end{split}
\end{equation*}

To estimate $F_{1}$, recall the fact that
\begin{equation*}
T(f_1,f_2)=T(f_1\chi_{\frac{6}{5}\tilde{B}},f_2\chi_{\frac{6}{5}\tilde{B}})+T(f_1\chi_{\frac{6}{5}\tilde{B}},f_2\chi_{X\backslash\frac{6}{5}\tilde{B}})
+T(f_1\chi_{X\backslash\frac{6}{5}\tilde{B}},f_2\chi_{\frac{6}{5}\tilde{B}})+T(f_1\chi_{X\backslash\frac{6}{5}\tilde{B}},f_2\chi_{X\backslash\frac{6}{5}\tilde{B}}),
\end{equation*}
then
\begin{equation*}
\begin{split}
F_{1}&\leq \biggl|m_{\tilde{B}}\biggl[T((b_{1}-m_{\tilde{B}}(b_{1}))f_{1},(b_{2}-m_{\tilde{B}}(b_{2}))f_{2})-
T((b_{1}-m_{\tilde{B}}(b_{1}))f_{1}\chi_{\frac{6}{5}\tilde{B}},(b_{2}-m_{\tilde{B}}(b_{2}))f_{2}\chi_{\frac{6}{5}\tilde{B}})\biggr]\\
&\ \ -m_{\tilde{B}}\biggl[T((b_{1}-m_{Q}(b_{1}))f_{1},(b_{2}-m_{Q}(b_{2}))f_{2})-
T((b_{1}-m_{Q}(b_{1}))f_{1}\chi_{\frac{6}{5}\tilde{B}},(b_{2}-m_{Q}(b_{2}))f_{2}\chi_{\frac{6}{5}\tilde{B}})\biggr]\biggr|\\
&\leq  \biggl|m_{\tilde{B}}\biggl[T((b_{1}-m_{\tilde{B}}(b_{1}))f_{1},(b_{2}-m_{\tilde{\tilde{B}}}(b_{2}))f_{2})
-T((b_{1}-m_{Q}(b_{1}))f_{1},(b_{2}-m_{Q}(b_{2}))f_{2})\biggr]\biggr|\\
&\ \ +\biggl|m_{\tilde{B}}\biggl[T((b_{1}-m_{\tilde{B}}(b_{1}))f_{1}\chi_{\frac{6}{5}\tilde{B}},(b_{2}-m_{\tilde{B}}(b_{2}))f_{2}\chi_{\frac{6}{5}\tilde{B}})\\
&\ \ -T((b_{1}-m_{Q}(b_{1}))f_{1}\chi_{\frac{6}{5}\tilde{B}},(b_{2}-m_{Q}(b_{2}))f_{2}\chi_{\frac{6}{5}\tilde{B}})\biggr]\biggr|\\
&=:F_{11}+F_{12}.
\end{split}
\end{equation*}

For $F_{11}$, recall the fact that
\begin{equation*}
\begin{split}
&T((b_{1}-m_{\tilde{B}}(b_{1}))f_{1},(b_{2}-m_{\tilde{B}}(b_{2}))f_{2})\\
=&T((b_{1}-b_{1}(z)+b_{1}(z)-m_{\tilde{B}}(b_{1}))f_{1},(b_{2}-b_{2}(z)+b_{2}(z)-m_{\tilde{B}}(b_{2}))f_{2})\\
=&(b_{1}(z)-m_{\tilde{B}}(b_{1}))(b_{2}(z)-m_{\tilde{B}}(b_{2}))T(f_{1},f_{2})\\
&\   +(b_{1}(z)-m_{\tilde{B}}(b_{1}))T(f_{1},(b_{2}-b_{2}(z))f_{2})\\
&\
+(b_{2}(z)-m_{\tilde{B}}(b_{2}))T((b_{1}-b_{1}(z))f_{1},f_{2})+T((b_{1}-b_{1}(z))f_{1},(b_{2}-b_{2}(z))f_{2})\\
=&(b_{1}(z)-m_{\tilde{B}}(b_{1}))(b_{2}(z)-m_{\tilde{B}}(b_{2}))T(f_{1},f_{2}) -(b_{1}(z)-m_{\tilde{B}}(b_{1}))[b_2,T](f_1,f_2)\\
&\ -(b_{2}(z)-m_{\tilde{B}}(b_{2}))[b_1,T](f_1,f_2)+[b_1,b_2,T](f_1,f_2).
\end{split}
\end{equation*}
Then
\begin{equation*}
\begin{split}
F_{11}\leq &\biggl|m_{\tilde{B}}[(b_{1}(z)-m_{\tilde{B}}(b_{1}))(b_{2}(z)-m_{\tilde{B}}(b_{2}))T(f_{1},f_{2})]\biggr|\\
&\ +\biggl|m_{\tilde{B}}[(b_{1}(z)-m_{\tilde{B}}(b_{1}))[b_2,T](f_1,f_2)]\biggr|\\
&\ +\biggl|m_{\tilde{B}}[(b_{2}(z)-m_{\tilde{B}}(b_{2}))[b_1,T](f_1,f_2)]\biggr|\\
&\ +\biggl|m_{\tilde{B}}[(b_{1}(z)-m_{Q}(b_{1}))(b_{2}(z)-m_{Q}(b_{2}))T(f_{1},f_{2})]\biggr|\\
&\ +\biggl|m_{\tilde{B}}[(b_{1}(z)-m_{Q}(b_{1}))[b_2,T](f_1,f_2)]\biggr|\\
&\ +\biggl|m_{\tilde{B}}[(b_{2}(z)-m_{Q}(b_{2}))[b_1,T](f_1,f_2)]\biggr|\\
=&:J_{1}+J_{2}+J_{3}+J_{4}+J_{5}+J_{6}.
\end{split}
\end{equation*}

Recall that $K_{\tilde{B},Q}\leq CK_{B,Q}$ with $B\subset \tilde{B} \subset Q$. Note that $\tilde{B}$ is a doubling ball, similar to estimate for (\ref{2.81}), (\ref{2.91}) and (\ref{2.92}) respcetively, we immediately obtain
\begin{equation*}
\begin{split}
J_{1}+J_{4}&\leq CK^{2}_{B,Q}||b_{1}||_{\ast}||b_{2}||_{\ast}M_{r,(6)}(T(f_{1},f_{2}))(x),\\
J_{2}+J_{5}&\leq CK_{B,Q}||b_{1}||_{\ast}M_{r,(6)}([b_{2},T](f_{1},f_{2}))(x),\\
J_{3}+J_{6}&\leq CK_{B,Q}||b_{2}||_{\ast}M_{r,(6)}([b_{1},T](f_{1},f_{2}))(x).\\
\end{split}
\end{equation*}

For $F_{12}$, with the same method to estimate $F_{11}$, we have
\begin{equation*}
\begin{split}
F_{12}&\leq CK^{2}_{B,Q}||b_{1}||_{\ast}||b_{2}||_{\ast}M_{r,(6)}(T(f_{1}\chi_{\frac{6}{5}\tilde{B}},f_{2}\chi_{\frac{6}{5}\tilde{B}}))(x)\\
&\ \ +CK_{B,Q}||b_{1}||_{\ast}M_{r,(6)}([b_{2},T](f_{1}\chi_{\frac{6}{5}\tilde{B}},f_{2}\chi_{\frac{6}{5}\tilde{B}}))(x)\\
&\ \ +CK_{B,Q}||b_{2}||_{\ast}M_{r,(6)}([b_{1},T](f_{1}\chi_{\frac{6}{5}\tilde{B}},f_{2}\chi_{\frac{6}{5}\tilde{B}}))(x).\\
\end{split}
\end{equation*}
Combining the estimates of $F_{11}$ and $F_{12}$, we complete the estimate for $F_{1}$.

 Now we turn to estimate $F_{2}$. By decomposing the region of the integral, we have

\begin{equation*}
\begin{split}
F_{2}\leq &|m_{\tilde{B}}\{T((b_{1}-m_{Q}(b_{1}))f_{1}\chi_{\frac{6}{5}\tilde{B}},(b_{2}-m_{Q}(b_{2}))f_{2}\chi_{6^{N}\tilde{B}\backslash\frac{6}{5}\tilde{B}})\}|\\
&+|m_{\tilde{B}}\{T((b_{1}-m_{Q}(b_{1}))f_{1}\chi_{6^{N}\tilde{B}\backslash\frac{6}{5}\tilde{B}},(b_{2}-m_{Q}(b_{2}))f_{2}\chi_{\frac{6}{5}\tilde{B}})\}|\\
&+|m_{\tilde{B}}\{T((b_{1}-m_{Q}(b_{1}))f_{1}\chi_{6^{N}\tilde{B}\backslash\frac{6}{5}\tilde{B}},(b_{2}-m_{Q}(b_{2}))f_{2}\chi_{6^{N}\tilde{B}\backslash\frac{6}{5}\tilde{B}})\}|\\
&+|m_{\tilde{B}}\{T((b_{1}-m_{Q}(b_{1}))f_{1}\chi_{6^{N}\tilde{B}},(b_{2}-m_{Q}(b_{2}))f_{2}\chi_{X\backslash 6^{N}\tilde{B}})\}\\
& -m_{Q}\{T((b_{1}-m_{Q}(b_{1}))f_{1}\chi_{6^{N}\tilde{B}},(b_{2}-m_{Q}(b_{2}))f_{2}\chi_{X\backslash6^{N}\tilde{B}})\}|\\
&+|m_{\tilde{B}}\{T((b_{1}-m_{Q}(b_{1}))f_{1}\chi_{X\backslash 6^{N}\tilde{B}},(b_{2}-m_{Q}(b_{2}))f_{2}\chi_{6^{N}\tilde{B}})\}\\
&-m_{Q}\{T((b_{1}-m_{Q}(b_{1}))f_{1}\chi_{X\backslash6^{N}\tilde{B}},(b_{2}-m_{Q}(b_{2}))f_{2}\chi_{6^{N}\tilde{B}})\}|\\
&+|m_{\tilde{B}}\{T((b_{1}-m_{Q}(b_{1}))f_{1}\chi_{X\backslash6^{N}\tilde{B}},(b_{2}-m_{Q}(b_{2}))f_{2}\chi_{X\backslash 6^{N}\tilde{B}})\}\\
&-m_{Q}\{T((b_{1}-m_{Q}(b_{1}))f_{1}\chi_{X\backslash6^{N}\tilde{B}},(b_{2}-m_{Q}(b_{2}))f_{2}\chi_{X\backslash 6^{N}\tilde{B}})\}|\\
&+|m_{Q}\{T((b_{1}-m_{Q}(b_{1}))f_{1}\chi_{\frac{6}{5}Q},(b_{2}-m_{Q}(b_{2}))f_{2}\chi_{6^{N}\tilde{B}\backslash\frac{6}{5}Q})\}|\\
&+|m_{Q}\{T((b_{1}-m_{Q}(b_{1}))f_{1}\chi_{6^{N}\tilde{B}\backslash\frac{6}{5}Q},(b_{2}-m_{Q}(b_{2}))f_{2}\chi_{\frac{6}{5}Q})\}|\\
&+|m_{Q}\{T((b_{1}-m_{Q}(b_{1}))f_{1}\chi_{6^{N}\tilde{B}\backslash\frac{6}{5}Q},(b_{2}-m_{Q}(b_{2}))f_{2}\chi_{6^{N}\tilde{B}\backslash\frac{6}{5}Q})\}|\\
=:&\sum_{i=1}^{9}F_{2i}.
\end{split}
\end{equation*}

For $F_{21}$, we first compute $T((b_{1}-m_{Q}(b_{1}))f_{1}\chi_{\frac{6}{5}\tilde{B}},(b_{2}-m_{Q}(b_{2}))f_{2}\chi_{6^{N}\tilde{B}\backslash\frac{6}{5}\tilde{B}})(z)$.
\begin{equation*}
\begin{split}
&|T((b_{1}-m_{Q}(b_{1}))f_{1}\chi_{\frac{6}{5}\tilde{B}},(b_{2}-m_{Q}(b_{2}))f_{2}\chi_{6^{N}\tilde{B}\backslash\frac{6}{5}\tilde{B}})(z)|\\
\leq& C\int_{6^{N}\tilde{B}\backslash\frac{6}{5}\tilde{B}}\int_{\frac{6}{5}\tilde{B}}
\frac{|(b_{1}(y_1)-m_{Q}(b_{1}))f_{1}(y_1)(b_{2}(y_2)-m_{Q}(b_{2}))f_{2}(y_2)|}{[\sum\limits_{i=1}^{2}\lambda(z,d(z,y_i))]^{2}}d\mu(y_1)d\mu(y_2)\\
\leq& C\int_{\frac{6}{5}\tilde{B}}\frac{|(b_{1}(y_1)-m_{Q}(b_{1}))f_{1}(y_1)|}{\lambda(z,d(z,y_1))}d\mu(y_1)
\biggl\{\sum_{k=1}^{N-1}\int_{6^{k+1}\tilde{B}\backslash6^{k}\tilde{B}}\frac{|(b_{2}(y_2)-m_{Q}(b_{2}))f_{2}(y_2)|}{\lambda(z,d(z,y_2))}d\mu(y_2)\\
&\ \ +\int_{6\tilde{B}\backslash\frac{6}{5}\tilde{B}}\frac{|(b_{2}(y_2)-m_{Q}(b_{2}))f_{2}(y_2)|}{\lambda(z,d(z,y_2))}d\mu(y_2)\biggr\}\\
\leq& C\frac{\mu(6\tilde{B})}{\lambda(x_{\tilde{B}},\frac{6}{5}r_{\tilde{B}})}\frac{1}{\mu(6\tilde{B})}
\int_{\frac{6}{5}\tilde{B}}|(b_{1}(y_1)-m_{Q}(b_{1}))f_{1}(y_1)|d\mu(y_1)\\
&\ \ \times\biggl\{ \sum_{k=1}^{N-1}\frac{\mu(5\times6^{k+1}\tilde{B})}{\lambda(x_{\tilde{B}},6^{k}r_{\tilde{B}})}
\frac{1}{\mu(5\times6^{k+1}\tilde{B})}\int_{6^{k+1}\tilde{B}}|(b_{2}(y_2)-m_{Q}(b_{2}))f_{2}(y_2)|d\mu(y_2)\\
&\ \ + \frac{\mu(5\times6\tilde{B})}{\lambda(x_{\tilde{B}},6r_{\tilde{B}})}
\frac{1}{\mu(5\times6\tilde{B})}\int_{6\tilde{B}}|(b_{2}(y_2)-m_{Q}(b_{2}))f_{2}(y_2)|d\mu(y_2)\biggr\}\\
\leq&CK^{3}_{\tilde{B},Q}||b_{1}||_{\ast}||b_{2}||_{\ast}M_{p_{1},(5)}f_{1}(x)M_{p_{2},(5)}f_{2}(x)\\
\leq&CK^{3}_{B,Q}||b_{1}||_{\ast}||b_{2}||_{\ast}M_{p_{1},(5)}f_{1}(x)M_{p_{2},(5)}f_{2}(x),
\end{split}
\end{equation*}
here we have used the property $K_{\tilde{B},Q}\leq CK_{B,Q}$ with $B\subset \tilde{B} \subset Q$.

Taking the mean over $z\in \tilde{B}$, we obtain
\begin{equation*}
F_{21}\leq CK^{3}_{B,Q}||b_{1}||_{\ast}||b_{2}||_{\ast}M_{p_{1},(5)}f_{1}(x)M_{p_{2},(5)}f_{2}(x).
\end{equation*}

Similar with the method to estimate $F_{21}$, we immediately have
\begin{equation*}
F_{22}+F_{27}+F_{28}\leq CK^{3}_{B,Q}||b_{1}||_{\ast}||b_{2}||_{\ast}M_{p_{1},(5)}f_{1}(x)M_{p_{2},(5)}f_{2}(x).
\end{equation*}

Now we estimate $F_{23}$.
\begin{equation*}
\begin{split}
&|T((b_{1}-m_{Q}(b_{1}))f_{1}\chi_{6^{N}\tilde{B}\backslash\frac{6}{5}\tilde{B}},(b_{2}-m_{Q}(b_{2}))f_{2}\chi_{6^{N}\tilde{B}\backslash\frac{6}{5}\tilde{B}})(z)\\
\leq& C\int_{6^{N}\tilde{B}\backslash\frac{6}{5}\tilde{B}} \int_{6^{N}\tilde{B}\backslash\frac{6}{5}\tilde{B}}
\frac{|(b_{1}(y_1)-m_{Q}(b_{1}))f_{1}(y_1)(b_{2}(y_2)-m_{Q}(b_{2}))f_{2}(y_2)|}{[\sum\limits_{i=1}^{2}\lambda(z,d(z,y_i))]^{2}}d\mu(y_1)d\mu(y_2)\\
\leq&C\sum_{j=1}^{N-1}\sum_{k=1}^{N-1}\int_{6^{j+1}\tilde{B}\backslash6^{j}\tilde{B}}\int_{6^{k+1}\tilde{B}\backslash6^{k}\tilde{B}}
\frac{|(b_{1}(y_1)-m_{Q}(b_{1}))f_{1}(y_1)(b_{2}(y_2)-m_{Q}(b_{2}))f_{2}(y_2)|}{[\sum\limits_{i=1}^{2}\lambda(z,d(z,y_i))]^{2}}d\mu(y_1)d\mu(y_2)\\
&\ \ +\int_{6\tilde{B}\backslash\frac{6}{5}\tilde{B}}\int_{6\tilde{B}\backslash\frac{6}{5}\tilde{B}}
\frac{|(b_{1}(y_1)-m_{Q}(b_{1}))f_{1}(y_1)(b_{2}(y_2)-m_{Q}(b_{2}))f_{2}(y_2)|}{[\sum\limits_{i=1}^{2}\lambda(z,d(z,y_i))]^{2}}d\mu(y_1)d\mu(y_2)\\
\leq&C\sum_{k=1}^{N-1}\frac{\mu(5\times 6^{k+1}\tilde{B})}{\lambda(x_{\tilde{B}},6^{k}r_{\tilde{B}})}\frac{1}{\mu(5\times 6^{k+1}\tilde{B})}\int_{6^{k+1}\tilde{B}}|(b_{1}(y_1)-m_{Q}(b_{1}))f_{1}(y_1)|d\mu(y_1)\\
&\ \ \times\sum_{j=1}^{N-1}\frac{\mu(5\times 6^{j+1}\tilde{B})}{\lambda(x_{\tilde{B}},6^{j}r_{\tilde{B}})}\frac{1}{\mu(5\times 6^{j+1}\tilde{B})}\int_{6^{j+1}\tilde{B}}|(b_{2}(y_2)-m_{Q}(b_{2}))f_{2}(y_2)|d\mu(y_2)\\
&\ \ +\biggl[\frac{\mu(30\tilde{B})}{\lambda(x_{\tilde{B}},6r_{\tilde{B}})}\biggr]^{2}\prod_{i=1}^{2}\frac{1}{\mu(30\tilde{B})}\int_{6\tilde{B}}|(b_{i}(y_i)-m_{Q}(b_{i}))f_{i}(y_i)|d\mu(y_i)\\
\leq &CK_{\tilde{B},Q}^{4}||b_{1}||_{\ast}||b_{2}||_{\ast}M_{p_{1},(5)}f_{1}(x)M_{p_{2},(5)}f_{2}(x).\\
\leq &CK_{B,Q}^{4}||b_{1}||_{\ast}||b_{2}||_{\ast}M_{p_{1},(5)}f_{1}(x)M_{p_{2},(5)}f_{2}(x).\\
\end{split}
\end{equation*}
Taking the mean over $z\in \tilde{B}$, we obtain
\begin{equation*}
F_{23}\leq CK_{B,Q}^{4}||b_{1}||_{\ast}||b_{2}||_{\ast}M_{p_{1},(5)}f_{1}(x)M_{p_{2},(5)}f_{2}(x).
\end{equation*}

For $F_{29}$, with the similar method to estimate $F_{23}$, we obtian
\begin{equation*}
F_{29}\leq CK_{B,Q}^{4}||b_{1}||_{\ast}||b_{2}||_{\ast}M_{p_{1},(5)}f_{1}(x)M_{p_{2},(5)}f_{2}(x).
\end{equation*}

Using the similar method to estimate $E_{42}(z)$, it follows
\begin{equation*}
F_{24}+F_{25}\leq
C||b_{1}||_{\ast}||b_{2}||_{\ast}M_{p_{1},(5)}f_{1}(x)M_{p_{2},(5)}f_{2}(x).
\end{equation*}

Using the similar method to estimate $E_{44}(z)$, we have
\begin{equation*}
F_{26}\leq
C||b_{1}||_{\ast}||b_{2}||_{\ast}M_{p_{1},(5)}f_{1}(x)M_{p_{2},(5)}f_{2}(x).
\end{equation*}

Next we consider $|h_{B}-h_{\tilde{B}}|$. With the similar method to estimate $F_{2}$, we easily obtain that
\begin{equation*}
\begin{split}
|h_{B}-h_{\tilde{B}}|
&=\biggl|m_{B}\biggl[T\biggl((b_{1}-
m_{\tilde{B}}(b_{1}))f_{1}\chi_{\frac{6}{5}B},(b_{2}-m_{\tilde{B}}(b_{2}))f_{2}\chi_{X\backslash\frac{6}{5}B}\biggr)\\
&\ \ +T\biggl((b_{1}-
m_{\tilde{B}}(b_{1}))f_{1}\chi_{X\backslash\frac{6}{5}B},(b_{2}-m_{\tilde{B}}(b_{2}))f_{2}\chi_{\frac{6}{5}B}\biggr)\\
&\ \ +T\biggl((b_{1}-
m_{\tilde{B}}(b_{1}))f_{1}\chi_{X\backslash\frac{6}{5}B},(b_{2}-m_{\tilde{B}}(b_{2}))f_{2}\chi_{X\backslash\frac{6}{5}B}\biggr)\biggr]\\
&\ \ -m_{\tilde{B}}\biggl[T\biggl((b_{1}-
m_{\tilde{B}}(b_{1}))f_{1}\chi_{\frac{6}{5}\tilde{B}},(b_{2}-m_{\tilde{B}}(b_{2}))f_{2}\chi_{X\backslash\frac{6}{5}\tilde{B}}\biggr)\\
&\ \ +T\biggl((b_{1}-
m_{\tilde{B}}(b_{1}))f_{1}\chi_{X\backslash\frac{6}{5}\tilde{B}},(b_{2}-m_{\tilde{B}}(b_{2}))f_{2}\chi_{\frac{6}{5}\tilde{B}}\biggr)\\
&\ \ +T\biggl((b_{1}-
m_{\tilde{B}}(b_{1}))f_{1}\chi_{X\backslash\frac{6}{5}\tilde{B}},(b_{2}-m_{\tilde{B}}(b_{2}))f_{2}\chi_{X\backslash\frac{6}{5}\tilde{B}}\biggr)\biggr]\biggr|\\
&\leq
C||b_{1}||_{\ast}||b_{2}||_{\ast}M_{p_{1},(5)}f_{1}(x)M_{p_{2},(5)}f_{2}(x).
\end{split}
\end{equation*}

Thus (\ref{2.18}) holds and hence (\ref{2.11}) is proved.
With the same method to prove (\ref{2.11}), we can obtain that (\ref{2.12}) and (\ref{2.13}) are
also hold. Here we omit the details. Thus Lemma \ref{lem2.5} in this case is proved.

{\bf Case 2: $0<l(B)=l<1$}. Recall that $\widetilde{B}$ is the smallest $(\alpha,\beta)$-doubling ball of
the form $6^{k}B$ with $k\in {\mathbf{N}}\bigcup\{0\}$.  Assume that $B_{0}$, $Q_{0}$ and $\tilde{B}_{0}$ are concentric with $B$, $Q$ and $\tilde{B}$ respectively and $l(B_{0})=l(B)^{\alpha}$, $l(Q_{0})=l(Q)^{\alpha}$, $l(\tilde{B}_{0})=l(\tilde{B})^{\alpha}$, then $B\subset B_{0}$, $Q\subset Q_{0}$ and $\tilde{B}\subset \tilde{B}_{0}$. Let $\rho$ be a number such that $6B_{0}=\rho B$. As in the proof of Theorem 9.1 in \cite{T2}, to obtain (\ref{2.14}), it suffices to show that
\begin{equation}\label{2.19}
\begin{split}
&\biggl(\frac{1}{\mu(6B_{0})}\int_{B}||[b_{1},b_{2},T](f_{1},f_{2})(z)|^{\delta}-|h_{B}|^{\delta}|d\mu(z)\biggr
)^{1/\delta}\\
\leq &C||b_{1}||_{\ast}||b_{2}||_{\ast}M_{r,(6)}(T(f_{1},f_{2}))(x)
+C||b_{1}||_{\ast}M_{r,(6)}([b_{2},T](f_{1},f_{2}))(x)\\
&+C||b_{2}||_{\ast}M_{r,(6)}([b_{1},T](f_{1},f_{2}))(x)+C||b_{1}||_{\ast}||b_{2}||_{\ast}M_{p_{1},(5)}f_{1}(x)M_{p_{2},(5)}f_{2}(x),
\end{split}
\end{equation}
 holds for
any $x$ and ball $B$ with $x\in B$, and
\begin{equation}\label{2.20}
\begin{split}
|h_{B}-h_{Q}|&\leq
CK_{B,Q}^{4}\biggr[||b_{1}||_{\ast}||b_{2}||_{\ast}M_{r,(6)}(T(f_{1},f_{2}))(x)
+||b_{1}||_{\ast}M_{r,(6)}([b_{2},T](f_{1},f_{2}))(x)\\
+&||b_{2}||_{\ast}M_{r,(6)}([b_{1},T](f_{1},f_{2}))(x)
+||b_{1}||_{\ast}||b_{2}||_{\ast}M_{p_{1},(5)}f_{1}(x)M_{p_{2},(5)}f_{2}(x)\\
+&||b_{1}||_{\ast}||b_{2}||_{\ast}M_{r,(6)}(T(f_{1}\chi_{\frac{6}{5}\tilde{B}_{0}},f_{2}\chi_{\frac{6}{5}\tilde{B}_{0}}))(x)
+||b_{1}||_{\ast}M_{r,(6)}([b_{2},T](f_{1}\chi_{\frac{6}{5}\tilde{B}_{0}},f_{2}\chi_{\frac{6}{5}\tilde{B}_{0}}))(x)\\
+&||b_{2}||_{\ast}M_{r,(6)}([b_{1},T](f_{1}\chi_{\frac{6}{5}\tilde{B}_{0}},f_{2}\chi_{\frac{6}{5}\tilde{B}_{0}}))(x)\biggr].
\end{split}
\end{equation}
for all balls $B\subset Q$ with $x\in B$, where $B$ is an arbitrary
ball, $Q$ is a doubling ball. For any ball $B$, we denote
\begin{equation*}
\begin{split}
h_{B}:=& m_{B}\biggl[T\biggl((b_{1}-
m_{\tilde{B}}(b_{1}))f_{1}\chi_{\frac{6}{5}B_{0}},(b_{2}-m_{\tilde{B}}(b_{2}))f_{2}\chi_{X\backslash\frac{6}{5}B_{0}}\biggr)\\
&\ \ +T\biggl((b_{1}-
m_{\tilde{B}}(b_{1}))f_{1}\chi_{X\backslash\frac{6}{5}B_{0}},(b_{2}-m_{\tilde{B}}(b_{2}))f_{2}\chi_{\frac{6}{5}B_{0}}\biggr)\\
&\ \ +T\biggl((b_{1}-
m_{\tilde{B}}(b_{1}))f_{1}\chi_{X\backslash\frac{6}{5}B_{0}},(b_{2}-m_{\tilde{B}}(b_{2}))f_{2}\chi_{X\backslash\frac{6}{5}B_{0}}\biggr)\biggr]
\end{split}
\end{equation*}
and
\begin{equation*}
\begin{split}
h_{Q}:=& m_{Q}\biggl[T\biggl((b_{1}-
m_{Q}(b_{1}))f_{1}\chi_{\frac{6}{5}Q_{0}},(b_{2}-m_{Q}(b_{2}))f_{2}\chi_{X\backslash\frac{6}{5}Q_{0}}\biggr)\\
&\ \ +T\biggl((b_{1}-
m_{Q}(b_{1}))f_{1}\chi_{X\backslash\frac{6}{5}Q_{0}},(b_{2}-m_{Q}(b_{2}))f_{2}\chi_{\frac{6}{5}Q_{0}}\biggr)\\
&\ \ +T\biggl((b_{1}-
m_{Q}(b_{1}))f_{1}\chi_{X\backslash\frac{6}{5}Q_{0}},(b_{2}-m_{Q}(b_{2}))f_{2}\chi_{X\backslash\frac{6}{5}Q_{0}}\biggr)\biggr].
\end{split}
\end{equation*}
Then
\begin{equation*}
\begin{split}
&\biggl(\frac{1}{\mu(6B_{0})}\int_{B}||[b_{1},b_{2},T](f_{1},f_{2})(z)|^{\delta}-|h_{B}|^{\delta}|d\mu(z)\biggr
)^{1/\delta}\\
\leq&
C\biggl(\frac{1}{\mu(6B_{0})}\int_{B}|[b_{1},b_{2},T](f_{1},f_{2})(z)-h_{B}|^{\delta}d\mu(z)\biggr)^{1/\delta}\\
\leq&
C\biggl(\frac{1}{\mu(6B_{0})}\int_{B}|(b_{1}(z)-m_{\tilde{B}}(b_{1}))(b_{2}(z)-m_{\tilde{B}}(b_{2}))T(f_{1},f_{2})(z)|^{\delta}d\mu(z)\biggr)^{1/\delta}\\
&\ \ +C\biggl(\frac{1}{\mu(6B_{0})}\int_{B}|(b_{1}(z)-m_{\tilde{B}}(b_{1}))T(f_{1},(b_{2}-b_{2}(z))f_{2})(z)|^{\delta}d\mu(z)\biggr)^{1/\delta}\\
&\ \ +C\biggl(\frac{1}{\mu(6B_{0})}\int_{B}|(b_{2}(z)-m_{\tilde{B}}(b_{2}))T((b_{1}-b_{1}(z))f_{1},f_{2})(z)|^{\delta}d\mu(z)\biggr)^{1/\delta}\\
&\ \
+C\biggl(\frac{1}{\mu(6B_{0})}\int_{B}|T((b_{1}-m_{\tilde{B}}(b_{1}))f_{1},(b_{2}-m_{\tilde{B}}(b_{2}))f_{2})(z)-h_{B}|^{\delta}d\mu(z)\biggr)^{1/\delta}\\
=&:G_{1}+G_{2}+G_{3}+G_{4}.
\end{split}
\end{equation*}

We first estimate $G_{1}$. Choosing $r_{1},r_{2}>1$ such that
$\dfrac{1}{r}+\dfrac{1}{r_{1}}+\dfrac{1}{r_{2}}=\dfrac{1}{\delta}$.
By H\"{o}lder's inequality, Lemma \ref{lem2.2} and the fact $B\subset B_{0}$, we have

\begin{equation*}
\begin{split}
G_{1}&\leq
C\biggl(\frac{1}{\mu(6B_{0})}\int_{B}|b_{1}(z)-m_{\tilde{B}}b_{1}|^{r_{1}}d\mu(z)\biggr)^{1/r_{1}}\\
&\ \ \times\biggl(\frac{1}{\mu(6B_{0})}\int_{B}|b_{2}(z)-m_{\tilde{B}}b_{2}|^{r_{2}}d\mu(z)\biggr)^{1/r_{2}}\\
&\ \ \times\biggl(\frac{1}{\mu(6B_{0})}\int_{B_{0}}|T(f_{1},f_{2})|^{r}d\mu(z)\biggr)^{1/r}\\
&\leq C||b_{1}||_{\ast}||b_{2}||_{\ast}M_{r,(6)}(T(f_{1},f_{2}))(x).
\end{split}
\end{equation*}

For $G_{2}$, let $s>1$ such that
$\dfrac{1}{s}+\dfrac{1}{r}=\dfrac{1}{\delta}$, by H\"{o}lder's
inequality, Lemma \ref{lem2.2} and the fact $B\subset B_{0}$, we deduce

\begin{equation*}
\begin{split}
G_{2}&\leq
C\biggl(\frac{1}{\mu(6B_{0})}\int_{B}|b_{1}(z)-m_{\tilde{B}}b_{1}|^{s}d\mu(z)\biggr)^{1/s}\\
&\ \ \times\biggl(\frac{1}{\mu(6B_{0})}\int_{B_{0}}|[b_{2},T](f_{1},f_{2})|^{r}d\mu(z)\biggr)^{1/r}\\
&\leq C||b_{1}||_{\ast}M_{r,(6)}([b_{2},T](f_{1},f_{2}))(x).
\end{split}
\end{equation*}

Similar to estimate $G_{2}$, we immediately obtain
\begin{equation*}
G_{3}\leq C||b_{2}||_{\ast}M_{r,(6)}([b_{1},T](f_{1},f_{2}))(x).
\end{equation*}

Let us turn to estimate $G_{4}$. Denote
$f_{j}^{1}=f_{j}\chi_{\frac{6}{5}B_{0}}$ and $f_{j}^{2}=f_{j}-f_{j}^{1}$
for $j=1,2$, we have
\begin{equation*}
\begin{split}
&|T((b_{1}-m_{\tilde{B}}(b_{1}))f_{1},(b_{2}-m_{\tilde{B}}(b_{2}))f_{2})(z)-h_{B}|\\
\leq&|T((b_{1}-m_{\tilde{B}}(b_{1}))f^{1}_{1},(b_{2}-m_{\tilde{B}}(b_{2}))f^{1}_{2})(z)|\\
&\ \ +|T((b_{1}-m_{\tilde{B}}(b_{1}))f^{1}_{1},(b_{2}-m_{\tilde{B}}(b_{2}))f^{2}_{2})(z)
-m_{B}[T((b_{1}-m_{\tilde{B}}(b_{1}))f^{1}_{1},(b_{2}-m_{\tilde{B}}(b_{2}))f^{2}_{2})]|\\
&\ \ +|T((b_{1}-m_{\tilde{B}}(b_{1}))f^{2}_{1},(b_{2}-m_{\tilde{B}}(b_{2}))f^{1}_{2})(z)
-m_{B}[T((b_{1}-m_{\tilde{B}}(b_{1}))f^{2}_{1},(b_{2}-m_{\tilde{B}}(b_{2}))f^{1}_{2})]|\\
&\ \ +|T((b_{1}-m_{\tilde{B}}(b_{1}))f^{2}_{1},(b_{2}-m_{\tilde{B}}(b_{2}))f^{2}_{2})(z)
-m_{B}[T((b_{1}-m_{\tilde{B}}(b_{1}))f^{2}_{1},(b_{2}-m_{\tilde{B}}(b_{2}))f^{2}_{2})]|\\
=&:G_{41}(z)+G_{42}(z)+G_{43}(z)+G_{44}(z).
\end{split}
\end{equation*}
Then
\begin{equation*}
\begin{split}
G_{4}&\leq
C\biggl(\frac{1}{\mu(6B_{0})}\int_{B}E_{41}(z)^{\delta}d\mu(z)\biggr)^{1/\delta}
+C\biggl(\frac{1}{\mu(6B_{0})}\int_{B}E_{42}(z)^{\delta}d\mu(z)\biggr)^{1/\delta}\\
&\ \ +
C\biggl(\frac{1}{\mu(6B_{0})}\int_{B}E_{43}(z)^{\delta}d\mu(z)\biggr)^{1/\delta}
+
C\biggl(\frac{1}{\mu(6B_{0})}\int_{B}E_{44}(z)^{\delta}d\mu(z)\biggr)^{1/\delta}\\
&=:G_{41}+G_{42}+G_{43}+G_{44}.
\end{split}
\end{equation*}

To estimate $G_{41}$, using Kolmogorov's
theorem, Lemma \ref{lem2.2} and H\"{o}lder's inequality, we obtain
\begin{equation*}
\begin{split}
G_{41}\leq&
C||T((b_{1}-m_{\tilde{B}}(b_{1}))f_{1}^{1},(b_{2}-m_{\tilde{B}}(b_{2}))f_{2}^{1})||_{L^{1/2,\infty}(\frac{6}{5}B_{0},\frac{d\mu(z)}{\mu(6B_{0})})}\\
\leq&
C\frac{1}{\mu(6B_{0})}\int_{\frac{6}{5}B_{0}}|(b_{1}(z)-m_{\tilde{B}}(b_{1}))f_{1}(z)|d\mu(z)\\
&\ \ \times\frac{1}{\mu(6B_{0})}\int_{\frac{6}{5}B_{0}}|(b_{2}(z)-m_{\tilde{B}}(b_{2}))f_{2}(z)|d\mu(z)\\
\leq&
C(\frac{1}{\mu(6B_{0})}\int_{\frac{6}{5}B_{0}}|b_{1}(z)-m_{\tilde{B}}(b_{1})|^{p'_{1}}d\mu(z))^{1/p'_{1}}\\
&\ \ \times(\frac{1}{\mu(6B_{0})}\int_{\frac{6}{5}B_{0}}|f_{1}(z)|^{p_{1}}d\mu(z))^{1/p_{1}}\\
&\ \ \times(\frac{1}{\mu(6B_{0})}\int_{\frac{6}{5}B_{0}}|b_{2}(z)-m_{\tilde{B}}(b_{2})|^{p'_{2}}d\mu(z))^{1/p'_{2}}\\
&\ \ \times(\frac{1}{\mu(6B_{0})}\int_{\frac{6}{5}B_{0}}|f_{2}(z)|^{p_{2}}d\mu(z))^{1/p_{2}}\\
\leq&
C||b_{1}||_{\ast}||b_{2}||_{\ast}M_{p_{1},(5)}f_{1}(x)M_{p_{2},(5)}f_{2}(x).
\end{split}
\end{equation*}

To compute $G_{42}$, let $z,y\in B$, $y_{1}\in \frac{6}{5}B_{0}$ and $y_{2}\in X\backslash \frac{6}{5}B_{0}$, then $\max\limits_{1\leq i\leq
2}d(z,y_{i})\geq d(z,y_{2})\geq Cl(B_{0})= Cl(B)^{\alpha}\geq Cd(z,y)^{\alpha}$. Using Definition \ref{def1.6}, Lemma \ref{lem2.2}, Lemma \ref{lem2.3},
H\"{o}lder's inequality and the properties of $\lambda$, we know

\begin{equation*}
\begin{split}
&|T((b_{1}-m_{\tilde{B}}(b_{1}))f^{1}_{1},(b_{2}-m_{\tilde{B}}(b_{2}))f^{2}_{2})(z)
-T((b_{1}-m_{\tilde{B}}(b_{1}))f^{1}_{1},(b_{2}-m_{\tilde{B}}(b_{2}))f^{2}_{2})(y)|\\
\leq& C\int_{X\backslash\frac{6}{5}B_{0}}\int_{\frac{6}{5}B_{0}}|K(z,y_1,y_2)-K(y,y_1,y_2)| |\prod_{i=1}^{2}(b_{i}(y_{i})-m_{\tilde{B}}(b_i))f_{i}(y_i)|d\mu(y_1)d\mu(y_2)\\
\leq &C\int_{X\backslash\frac{6}{5}B_{0}}\int_{\frac{6}{5}B_{0}}\frac{d(z,y)^{\delta}}
{\biggl[\sum\limits_{i=1}^{2}d(z,y_{i})\biggr]^{\delta/\alpha}\biggl[\sum\limits_{i=1}^{2}\lambda(z,d(z,y_{i}))\biggr]^{2}} \\ &\ \ \ \times|\prod_{i=1}^{2}(b_{i}(y_{i})-m_{\tilde{B}}(b_i))f_{i}(y_i)|d\mu(y_1)d\mu(y_2)\\
\leq& C\int_{\frac{6}{5}B_{0}}
\frac{|b_{1}(y_{1})-m_{\tilde{B}}(b_{1})||f_{1}(y_{1})|}{\lambda(z,d(z,y_{1}))}d\mu(y_{1})\\
&\ \ \ \times\int_{X\backslash\frac{6}{5}B_{0}}
\frac{d(z,y)^{\delta}}
{d(z,y_{2})^{\delta/\alpha}}\frac{|b_{2}(y_{2})-m_{\tilde{B}}(b_{2})||f_{2}(y_{2})|}{\lambda(z,d(z,y_{2}))}d\mu(y_{2})\\
\leq& C\frac{\mu(6B_{0})}{\lambda(x_{B},\frac{6}{5}r_{B_{0}})}\frac{1}{\mu(6B_{0})}\int_{\frac{6}{5}B_{0}}
|b_{1}(y_{1})-m_{\tilde{B}}(b_{1})||f_{1}(y_{1})|d\mu(y_{1})\\
&\ \ \ \times\sum_{k=1}^{\infty}\int_{6^{k}\frac{6}{5}B_{0}\backslash6^{k-1}\frac{6}{5}B_{0}}
\frac{d(z,y)^{\delta}}
{d(z,y_{2})^{\delta/\alpha}}\frac{|b_{2}(y_{2})-m_{\tilde{B}}(b_{2})||f_{2}(y_{2})|}{\lambda(z,d(z,y_{2}))}d\mu(y_{2})\\
\end{split}
\end{equation*}

\begin{equation*}
\begin{split}
&\leq
C(\frac{1}{\mu(6B_{0})}\int_{\frac{6}{5}B_{0}}|b_{1}(y_{1})-m_{\tilde{B}}(b_{1})|^{p'_{1}}d\mu(y_{1}))^{1/p'_{1}}\\
&\ \ \ \times(\frac{1}{\mu(6B_{0})}\int_{\frac{6}{5}B_{0}}|f_{1}(y_{1})|^{p_{1}}d\mu(y_{1}))^{1/p_{1}}\\
&\ \ \ \times
\sum_{k=1}^{\infty}\frac{1}{\lambda(x_{B},6^{k-1}\frac{6}{5}r_{B_{0}})}6^{-k\delta/\alpha}
\int_{6^{k}\frac{6}{5}B_{0}}|b_{2}(y_{2})-m_{\tilde{B}}
(b_{2})||f_{2}(y_{2})|d\mu(y_{2})\\
&\leq
C||b_{1}||_{\ast}M_{p_{1},(5)}f_{1}(x)\sum_{k=1}^{\infty}6^{-k\delta/\alpha}\biggl[\biggl(\frac{1}{\mu(5\times6^{k}\frac{6}{5}B_{0})}\\
&\ \ \ \times\int_{6^{k}\frac{6}{5}B_{0}}|b_{2}(y_{2})-m_{\widetilde{6^{k}\frac{6}{5}B_{0}}}(b_{2})|^{p'_{2}}d\mu(y_{2})\biggr)^{1/p'_{2}}\\
&\ \ \ \times
\biggl(\frac{1}{\mu(5\times6^{k}\frac{6}{5}B_{0})}\int_{6^{k}\frac{6}{5}B_{0}}|f_{2}(y_{2})|^{p_{2}}d\mu(y_{2})\biggr)^{1/p_{2}}\\
&\ \ \ +C
k||b_{2}||_{\ast}\frac{1}{\mu(5\times6^{k}\frac{6}{5}B_{0})}\int_{6^{k}\frac{6}{5}B_{0}}|f_{2}(y_{2})|d\mu(y_{2})\biggr]\\
&\leq
C||b_{1}||_{\ast}||b_{2}||_{\ast}M_{p_{1},(5)}f_{1}(x)M_{p_{2},(5)}f_{2}(x).\\
\end{split}
\end{equation*}
Taking the mean over $y\in B$, we have
\begin{equation*}
G_{42}(z)\leq
C||b_{1}||_{\ast}||b_{2}||_{\ast}M_{p_{1},(5)}f_{1}(x)M_{p_{2},(5)}f_{2}(x).
\end{equation*}
Thus
\begin{equation*}
G_{42}\leq
C||b_{1}||_{\ast}||b_{2}||_{\ast}M_{p_{1},(5)}f_{1}(x)M_{p_{2},(5)}f_{2}(x).
\end{equation*}

Similarly, we get
\begin{equation*}
G_{43}\leq
C||b_{1}||_{\ast}||b_{2}||_{\ast}M_{p_{1},(5)}f_{1}(x)M_{p_{2},(5)}f_{2}(x).
\end{equation*}

For $G_{44}$, by Definition \ref{def1.6}, Lemma \ref{lem2.2}, Lemma \ref{lem2.3},
H\"{o}lder's inequality and the properties of $\lambda$, we obtain
\begin{equation*}
\begin{split}
&|T((b_{1}-m_{\tilde{B}}(b_{1}))f^{2}_{1},(b_{2}-m_{\tilde{B}}(b_{2}))f^{2}_{2})(z)
-T((b_{1}-m_{\tilde{B}}(b_{1}))f^{2}_{1},(b_{2}-m_{\tilde{B}}(b_{2}))f^{2}_{2})(y)|\\
\leq &C\int_{X\backslash
\frac{6}{5}B_{0}}\int_{X\backslash\frac{6}{5}B_{0}}|K(z,y_{1},y_{2})-K(y,y_{1},y_{2})|\\
&\ \ \ \times|\prod_{i=1}^{2}(b_{i}(y_{i})-m_{\tilde{B}}b_{i})f_{i}(y_{i})|d\mu(y_{1})d\mu(y_{2})\\
\leq &C\int_{X\backslash
\frac{6}{5}B_{0}}\int_{X\backslash\frac{6}{5}B_{0}}\frac{d(z,y)^{\delta}
|\prod_{i=1}^{2}(b_{i}(y_{i})-m_{\tilde{B}}b_{i})f_{i}(y_{i})|d\mu(y_{1})d\mu(y_{2})}
{(d(z,y_{1})+d(z,y_{2}))^{\delta/\alpha}[\sum_{j=1}^{2}\lambda(z,d(z,y_{j}))]^{2}}\\
\leq &C\prod_{i=1}^{2}\int_{X\backslash
\frac{6}{5}B_{0}}\frac{d(z,y)^{\delta_{i}}
|b_{i}(y_{i})-m_{\tilde{B}}b_{i}||f_{i}(y_{i})|d\mu(y_{i})}{d(z,y_{i})^{\delta_{i}/\alpha}\lambda(x_B,d(z,y_{i}))}\\
\leq
&C\prod_{i=1}^{2}\sum_{k=1}^{\infty}\int_{6^{k}\frac{6}{5}B_{0}}6^{-k\delta_{i}/\alpha}\frac{\mu(5\times6^{k}\frac{6}{5}B_{0})}
{\lambda(x_{B},5\times6^{k}\frac{6}{5}r_{B_{0}})}\\
&\ \ \times\frac{1}{\mu(5\times6^{k}\frac{6}{5}B_{0})}|b_{i}(y_{i})-m_{\tilde{B}}(b_{i})||f_{i}(y_i)|d\mu(y_{i})\\
\leq
&C\prod_{i=1}^{2}\sum_{k=1}^{\infty}6^{-k\delta_{i}/\alpha}(\frac{1}{\mu(5\times6^{k}\frac{6}{5}B_{0})}
\int_{6^{k}\frac{6}{5}B_{0}}|b_{i}(y_{i})-m_{\tilde{B}}(b_{i})|^{p'_{i}}d\mu(y_{i}))^{1/p'_{i}}\\
&\ \ \ \times(\frac{1}{\mu(5\times6^{k}\frac{6}{5}B_{0})}
\int_{6^{k}\frac{6}{5}B_{0}}|f_{i}(y_i)|^{p_{i}})^{1/p_{i}}\\
\leq
&C\prod_{i=1}^{2}\sum_{k=1}^{\infty}6^{-k\delta_{i}/\alpha}M_{p_{i},(5)}f_{i}(x)(\frac{1}{\mu(5\times6^{k}\frac{6}{5}B_{0})}
\int_{6^{k}\frac{6}{5}B_{0}}|b_{i}(y_{i})-m_{\widetilde{6^{k}\frac{6}{5}B_{0}}}(b_i)\\
&\ \ \ +m_{\widetilde{6^{k}\frac{6}{5}B_{0}}}(b_{i})-m_{\tilde{B}}(b_{i})|^{p'_{i}}d\mu(y_{i}))^{1/p'_{i}}\\
\leq
&C\prod_{i=1}^{2}\sum_{k=1}^{\infty}6^{-k\delta_{i}/\alpha}k||b_{i}||_{\ast}M_{p_{i},(5)}f_{i}(x)\\
\leq&
C||b_{1}||_{\ast}||b_{2}||_{\ast}M_{p_{1},(5)}f_{1}(x)M_{p_{2},(5)}f_{2}(x).
\end{split}
\end{equation*}
where $\delta_{1},\delta_{2}>0$ and $\delta_{1}+\delta_{2}=\delta$.

Taking the mean over $y\in B$, then
\begin{equation*}
G_{44}(z)\leq
C||b_{1}||_{\ast}||b_{2}||_{\ast}M_{p_{1},(5)}f_{1}(x)M_{p_{2},(5)}f_{2}(x).
\end{equation*}
Therefore
\begin{equation*}
G_{44}\leq
C||b_{1}||_{\ast}||b_{2}||_{\ast}M_{p_{1},(5)}f_{1}(x)M_{p_{2},(5)}f_{2}(x).
\end{equation*}
So (\ref{2.19}) can be obtained.

Next we prove (\ref{2.20}). Consider two balls $B\subset Q$ with $x\in B$,
where $B$ is an arbitrary ball and $Q$ is a doubling ball. Assume that $B_{0}$, $Q_{0}$ and $\tilde{B}_{0}$ are concentric with $B$, $Q$ and $\tilde{B}$ respectively and $l(B_{0})=l(B)^{\alpha}$, $l(Q_{0})=l(Q)^{\alpha}$, $l(\tilde{B}_{0})=l(\tilde{B})^{\alpha}$, then $B\subset B_{0}$, $Q\subset Q_{0}$ and $\tilde{B}\subset \tilde{B}_{0}$. Thus we have the four cases: $B\subset Q\subset Q_{0}\subset \tilde{B}\subset \tilde{B}_{0}$, $B\subset Q\subset \tilde{B}\subset Q_{0}\subset \tilde{B}_{0}$, $B\subset\tilde{B} \subset\tilde{B}_{0}\subset Q\subset Q_{0}$,
$B\subset\tilde{B} \subset Q\subset\tilde{B}_{0}\subset Q_{0}$. Without loss of generality, we only consider $B\subset\tilde{B} \subset\tilde{B}_{0}\subset Q\subset Q_{0}$ in this paper.  The other cases can be considered by the same method.

Let
$N=N_{\tilde{B}_{0}, Q_{0}}+1$. It is easy to see that
\begin{equation*}
|h_{B}-h_{Q}|\leq |h_{B}-h_{\tilde{B}}|+|h_{\tilde{B}}-h_{Q}|.
\end{equation*}
Now we first consider $|h_{\tilde{B}}-h_{Q}|$. Write
\begin{equation*}
\begin{split}
|h_{\tilde{B}}-h_{Q}|\leq & \biggl|m_{\tilde{B}}\biggl[T\biggl((b_{1}-
m_{\tilde{B}}(b_{1}))f_{1}\chi_{\frac{6}{5}\tilde{B}_{0}},(b_{2}-m_{\tilde{B}}(b_{2}))f_{2}\chi_{X\backslash\frac{6}{5}\tilde{B}_{0}}\biggr)\\
&\ \ +T\biggl((b_{1}-
m_{\tilde{B}}(b_{1}))f_{1}\chi_{X\backslash\frac{6}{5}\tilde{B}_{0}},(b_{2}-m_{\tilde{B}}(b_{2}))f_{2}\chi_{\frac{6}{5}\tilde{B}_{0}}\biggr)\\
&\ \ +T\biggl((b_{1}-
m_{\tilde{B}}(b_{1}))f_{1}\chi_{X\backslash\frac{6}{5}\tilde{B}_{0}},(b_{2}-m_{\tilde{B}}(b_{2}))f_{2}\chi_{X\backslash\frac{6}{5}\tilde{B}_{0}}\biggr)\biggr]\\
&\ \ -m_{\tilde{B}}\biggl[T\biggl((b_{1}-
m_{Q}(b_{1}))f_{1}\chi_{\frac{6}{5}\tilde{B}_{0}},(b_{2}-m_{Q}(b_{2}))f_{2}\chi_{X\backslash\frac{6}{5}\tilde{B}_{0}}\biggr)\\
&\ \ +T\biggl((b_{1}-
m_{Q}(b_{1}))f_{1}\chi_{X\backslash\frac{6}{5}\tilde{B}_{0}},(b_{2}-m_{Q}(b_{2}))f_{2}\chi_{\frac{6}{5}\tilde{B}_{0}}\biggr)\\
&\ \ +T\biggl((b_{1}-
m_{Q}(b_{1}))f_{1}\chi_{X\backslash\frac{6}{5}\tilde{B}_{0}},(b_{2}-m_{Q}(b_{2}))f_{2}\chi_{X\backslash\frac{6}{5}\tilde{B}_{0}}\biggr)\biggr]\biggr|\\
&\ \ +\biggl|m_{\tilde{B}}\biggl[T\biggl((b_{1}-
m_{Q}(b_{1}))f_{1}\chi_{\frac{6}{5}\tilde{B}_{0}},(b_{2}-m_{Q}(b_{2}))f_{2}\chi_{X\backslash\frac{6}{5}\tilde{B}_{0}}\biggr)\\
&\ \ +T\biggl((b_{1}-
m_{Q}(b_{1}))f_{1}\chi_{X\backslash\frac{6}{5}\tilde{B}_{0}},(b_{2}-m_{Q}(b_{2}))f_{2}\chi_{\frac{6}{5}\tilde{B}_{0}}\biggr)\\
&\ \ +T\biggl((b_{1}-
m_{Q}(b_{1}))f_{1}\chi_{X\backslash\frac{6}{5}\tilde{B}_{0}},(b_{2}-m_{Q}(b_{2}))f_{2}\chi_{X\backslash\frac{6}{5}\tilde{B}_{0}}\biggr)\biggr]\\
&\ \ -m_{Q}\biggl[T\biggl((b_{1}-
m_{Q}(b_{1}))f_{1}\chi_{\frac{6}{5}Q_{0}},(b_{2}-m_{Q}(b_{2}))f_{2}\chi_{X\backslash\frac{6}{5}Q_{0}}\biggr)\\
&\ \ +T\biggl((b_{1}-
m_{Q}(b_{1}))f_{1}\chi_{X\backslash\frac{6}{5}Q_{0}},(b_{2}-m_{Q}(b_{2}))f_{2}\chi_{\frac{6}{5}Q_{0}}\biggr)\\
&\ \ +T\biggl((b_{1}-
m_{Q}(b_{1}))f_{1}\chi_{X\backslash\frac{6}{5}Q_{0}},(b_{2}-m_{Q}(b_{2}))f_{2}\chi_{X\backslash\frac{6}{5}Q_{0}}\biggr)\biggr]\biggr|\\
=:&H_{1}+H_{2}.
\end{split}
\end{equation*}

To estimate $H_{1}$, write
\begin{equation*}
\begin{split}
H_{1}&\leq \biggl|m_{\tilde{B}}\biggl[T((b_{1}-m_{\tilde{B}}(b_{1}))f_{1},(b_{2}-m_{\tilde{B}}(b_{2}))f_{2})-
T((b_{1}-m_{\tilde{B}}(b_{1}))f_{1}\chi_{\frac{6}{5}\tilde{B}_{0}},(b_{2}-m_{\tilde{B}}(b_{2}))f_{2}\chi_{\frac{6}{5}\tilde{B}_{0}})\biggr]\\
&\ \ -m_{\tilde{B}}\biggl[T((b_{1}-m_{Q}(b_{1}))f_{1},(b_{2}-m_{Q}(b_{2}))f_{2})-
T((b_{1}-m_{Q}(b_{1}))f_{1}\chi_{\frac{6}{5}\tilde{B}_{0}},(b_{2}-m_{Q}(b_{2}))f_{2}\chi_{\frac{6}{5}\tilde{B}_{0}})\biggr]\biggr|\\
&\leq  \biggl|m_{\tilde{B}}\biggl[T((b_{1}-m_{\tilde{B}}(b_{1}))f_{1},(b_{2}-m_{\tilde{B}}(b_{2}))f_{2})
-T((b_{1}-m_{Q}(b_{1}))f_{1},(b_{2}-m_{Q}(b_{2}))f_{2})\biggr]\biggr|\\
&\ \ +\biggl|m_{\tilde{B}}\biggl[T((b_{1}-m_{\tilde{B}}(b_{1}))f_{1}\chi_{\frac{6}{5}\tilde{B}_{0}},(b_{2}-m_{\tilde{B}}(b_{2}))f_{2}\chi_{\frac{6}{5}\tilde{B}_{0}})\\
&\ \ -T((b_{1}-m_{Q}(b_{1}))f_{1}\chi_{\frac{6}{5}\tilde{B}_{0}},(b_{2}-m_{Q}(b_{2}))f_{2}\chi_{\frac{6}{5}\tilde{B}_{0}})\biggr]\biggr|\\
&=:H_{11}+H_{12}.
\end{split}
\end{equation*}

For $H_{11}$, since it equals to $F_{11}$, we immediately have

\begin{equation*}
\begin{split}
H_{11}&\leq C||b_{1}||_{\ast}||b_{2}||_{\ast}M_{r,(6)}(T(f_{1},f_{2}))(x)+C||b_{1}||_{\ast}M_{r,(6)}([b_{2},T](f_{1},f_{2}))(x)\\
&\ \ +C||b_{2}||_{\ast}M_{r,(6)}([b_{1},T](f_{1},f_{2}))(x).\\
\end{split}
\end{equation*}

For $H_{12}$, with the same method to estimate $F_{11}$, we have
\begin{equation*}
\begin{split}
H_{12}&\leq CK^{2}_{B,Q}||b_{1}||_{\ast}||b_{2}||_{\ast}M_{r,(6)}(T(f_{1}\chi_{\frac{6}{5}\tilde{B}_{0}},f_{2}\chi_{\frac{6}{5}\tilde{B}_{0}}))(x)\\
&\ \ +CK_{B,Q}||b_{1}||_{\ast}M_{r,(6)}([b_{2},T](f_{1}\chi_{\frac{6}{5}\tilde{B}_{0}},f_{2}\chi_{\frac{6}{5}\tilde{B}_{0}}))(x)\\
&\ \ +CK_{B,Q}||b_{2}||_{\ast}M_{r,(6)}([b_{1},T](f_{1}\chi_{\frac{6}{5}\tilde{B}_{0}},f_{2}\chi_{\frac{6}{5}\tilde{B}_{0}}))(x).\\
\end{split}
\end{equation*}
Combining the estimates of $H_{11}$ and $H_{12}$, we complete the estimate for $H_{1}$.

 Now we turn to estimate $H_{2}$. By decomposing the region of the integral, we have
\begin{equation*}
\begin{split}
H_{2}\leq &|m_{\tilde{B}}\{T((b_{1}-m_{Q}(b_{1}))f_{1}\chi_{\frac{6}{5}\tilde{B}_{0}},(b_{2}-m_{Q}(b_{2}))f_{2}\chi_{6^{N}\tilde{B}_{0}\backslash\frac{6}{5}\tilde{B}_{0}})\}|\\
&+|m_{\tilde{B}}\{T((b_{1}-m_{Q}(b_{1}))f_{1}\chi_{6^{N}\tilde{B}_{0}\backslash\frac{6}{5}\tilde{B}_{0}},(b_{2}-m_{Q}(b_{2}))f_{2}\chi_{\frac{6}{5}\tilde{B}_{0}})\}|\\
&+|m_{\tilde{B}}\{T((b_{1}-m_{Q}(b_{1}))f_{1}\chi_{6^{N}\tilde{B}_{0}\backslash\frac{6}{5}\tilde{B}_{0}},(b_{2}-m_{Q}(b_{2}))f_{2}\chi_{6^{N}\tilde{B}_{0}\backslash\frac{6}{5}\tilde{B}_{0}})\}|\\
&+|m_{\tilde{B}}\{T((b_{1}-m_{Q}(b_{1}))f_{1}\chi_{6^{N}\tilde{B}_{0}},(b_{2}-m_{Q}(b_{2}))f_{2}\chi_{X\backslash 6^{N}\tilde{B}_{0}})\}\\
& -m_{Q}\{T((b_{1}-m_{Q}(b_{1}))f_{1}\chi_{6^{N}\tilde{B}_{0}},(b_{2}-m_{Q}(b_{2}))f_{2}\chi_{X\backslash6^{N}\tilde{B}_{0}})\}|\\
&+|m_{\tilde{B}}\{T((b_{1}-m_{Q}(b_{1}))f_{1}\chi_{X\backslash 6^{N}\tilde{B}_{0}},(b_{2}-m_{Q}(b_{2}))f_{2}\chi_{6^{N}\tilde{B}_{0}})\}\\
&-m_{Q}\{T((b_{1}-m_{Q}(b_{1}))f_{1}\chi_{X\backslash6^{N}\tilde{B}_{0}},(b_{2}-m_{Q}(b_{2}))f_{2}\chi_{6^{N}\tilde{B}_{0}})\}|\\
&+|m_{\tilde{B}}\{T((b_{1}-m_{Q}(b_{1}))f_{1}\chi_{X\backslash6^{N}\tilde{B}_{0}},(b_{2}-m_{Q}(b_{2}))f_{2}\chi_{X\backslash 6^{N}\tilde{B}_{0}})\}\\
&-m_{Q}\{T((b_{1}-m_{Q}(b_{1}))f_{1}\chi_{X\backslash6^{N}\tilde{B}_{0}},(b_{2}-m_{Q}(b_{2}))f_{2}\chi_{X\backslash 6^{N}\tilde{B}_{0}})\}|\\
&+|m_{Q}\{T((b_{1}-m_{Q}(b_{1}))f_{1}\chi_{\frac{6}{5}Q_{0}},(b_{2}-m_{Q}(b_{2}))f_{2}\chi_{6^{N}B_{0}\backslash\frac{6}{5}Q_{0}})\}|\\
&+|m_{Q}\{T((b_{1}-m_{Q}(b_{1}))f_{1}\chi_{6^{N}\tilde{B}_{0}\backslash\frac{6}{5}Q_{0}},(b_{2}-m_{Q}(b_{2}))f_{2}\chi_{\frac{6}{5}Q_{0}})\}|\\
&+|m_{Q}\{T((b_{1}-m_{Q}(b_{1}))f_{1}\chi_{6^{N}\tilde{B}_{0}\backslash\frac{6}{5}Q_{0}},(b_{2}-m_{Q}(b_{2}))f_{2}\chi_{6^{N}\tilde{B}_{0}\backslash\frac{6}{5}Q_{0}})\}|\\
=:&\sum_{i=1}^{9}H_{2i}.
\end{split}
\end{equation*}

Similar to estimate $F_{2i},1\leq i\leq 9$, we obtain that

\begin{equation*}
\begin{split}
H_{21}+H_{22}+H_{27}+H_{28}&\leq CK^{3}_{B,Q}||b_{1}||_{\ast}||b_{2}||_{\ast}M_{p_{1},(5)}f_{1}(x)M_{p_{2},(5)}f_{2}(x),\\
H_{23}+H_{29}&\leq CK_{B,Q}^{4}||b_{1}||_{\ast}||b_{2}||_{\ast}M_{p_{1},(5)}f_{1}(x)M_{p_{2},(5)}f_{2}(x),\\
H_{24}+H_{25}+H_{26}&\leq C||b_{1}||_{\ast}||b_{2}||_{\ast}M_{p_{1},(5)}f_{1}(x)M_{p_{2},(5)}f_{2}(x).\\
\end{split}
\end{equation*}

For simplicity, we only computer $H_{21}$. We write
\begin{equation*}
\begin{split}
&|T((b_{1}-m_{Q}(b_{1}))f_{1}\chi_{\frac{6}{5}\tilde{B}_{0}},(b_{2}-m_{Q}(b_{2}))f_{2}\chi_{6^{N}\tilde{B}_{0}\backslash\frac{6}{5}\tilde{B}_{0}})(z)|\\
\leq& C\int_{6^{N}\tilde{B}_{0}\backslash\frac{6}{5}\tilde{B}_{0}}\int_{\frac{6}{5}\tilde{B}_{0}}
\frac{|(b_{1}(y_1)-m_{Q}(b_{1}))f_{1}(y_1)(b_{2}(y_2)-m_{Q}(b_{2}))f_{2}(y_2)|}{[\sum\limits_{i=1}^{2}\lambda(z,d(z,y_i))]^{2}}d\mu(y_1)d\mu(y_2)\\
\leq& C\int_{\frac{6}{5}\tilde{B}_{0}}\frac{|(b_{1}(y_1)-m_{Q}(b_{1}))f_{1}(y_1)|}{\lambda(z,d(z,y_1))}d\mu(y_1)
\biggl\{\sum_{k=1}^{N-1}\int_{6^{k+1}\tilde{B}_{0}\backslash6^{k}\tilde{B}_{0}}\frac{|(b_{2}(y_2)-m_{Q}(b_{2}))f_{2}(y_2)|}{\lambda(z,d(z,y_2))}d\mu(y_2)\\
&\ \ +\int_{6\tilde{B}_{0}\backslash\frac{6}{5}\tilde{B}_{0}}\frac{|(b_{2}(y_2)-m_{Q}(b_{2}))f_{2}(y_2)|}{\lambda(z,d(z,y_2))}d\mu(y_2)\biggr\}\\
\leq& C\frac{\mu(6\tilde{B}_{0})}{\lambda(x_{\tilde{B}_{0}},\frac{6}{5}r_{\tilde{B}_{0}})}\frac{1}{\mu(6\tilde{B}_{0})}
\int_{\frac{6}{5}\tilde{B}_{0}}|(b_{1}(y_1)-m_{Q}(b_{1}))f_{1}(y_1)|d\mu(y_1)\\
&\ \ \times\biggl\{ \sum_{k=1}^{N-1}\frac{\mu(5\times6^{k+1}\tilde{B}_{0})}{\lambda(x_{\tilde{B}_{0}},6^{k}r_{\tilde{B}_{0}})}
\frac{1}{\mu(5\times6^{k+1}\tilde{B}_{0})}\int_{6^{k+1}\tilde{B}_{0}}|(b_{2}(y_2)-m_{Q}(b_{2}))f_{2}(y_2)|d\mu(y_2)\\
&\ \ + \frac{\mu(5\times6\tilde{B}_{0})}{\lambda(x_{\tilde{B}_{0}},6r_{\tilde{B}_{0}})}
\frac{1}{\mu(5\times6\tilde{B}_{0})}\int_{6\tilde{B}_{0}}|(b_{2}(y_2)-m_{Q}(b_{2}))f_{2}(y_2)|d\mu(y_2)\biggr\}\\
\leq&CK^{2}_{\tilde{B}_{0},Q}K_{\tilde{B}_{0},Q_{0}}||b_{1}||_{\ast}||b_{2}||_{\ast}M_{p_{1},(5)}f_{1}(x)M_{p_{2},(5)}f_{2}(x)\\
\leq&CK^{3}_{B,Q}||b_{1}||_{\ast}||b_{2}||_{\ast}M_{p_{1},(5)}f_{1}(x)M_{p_{2},(5)}f_{2}(x),
\end{split}
\end{equation*}
here we have used the conclusions $K_{\tilde{B}_{0},Q}\leq CK_{B,Q}$ and $K_{\tilde{B}_{0},Q_{0}}\leq CK_{B,Q}$, which can be obtained by Lemma 2.4.

Taking the mean over $z\in \tilde{B}$, we obtain
\begin{equation*}
H_{21}\leq CK^{3}_{B,Q}||b_{1}||_{\ast}||b_{2}||_{\ast}M_{p_{1},(5)}f_{1}(x)M_{p_{2},(5)}f_{2}(x).
\end{equation*}

Combining the estimate for $H_{2i},1\leq i\leq 9$, then
\begin{equation*}
H_{2}\leq
CK_{B,Q}^{4}||b_{1}||_{\ast}||b_{2}||_{\ast}M_{p_{1},(5)}f_{1}(x)M_{p_{2},(5)}f_{2}(x).
\end{equation*}

Next we consider $|h_{B}-h_{\tilde{B}}|$. With the similar method to estimate $H_{2}$, we easily obtain that
\begin{equation*}
|h_{B}-h_{\tilde{B}}|\leq
C||b_{1}||_{\ast}||b_{2}||_{\ast}M_{p_{1},(5)}f_{1}(x)M_{p_{2},(5)}f_{2}(x).
\end{equation*}

Thus (\ref{2.20}) holds and hence (\ref{2.14}) is proved.
With the same method to prove (\ref{2.14}), we can obtain that (\ref{2.15}) and (\ref{2.16}) also hold.
Here we omit the details. Thus Lemma \ref{lem2.5} has been proved.
\end{proof}

\begin{proof}[Proof of Theorem \ref{main-thm}]
 Let $0<\delta<1/2$, $1<p_{1},p_{2},q<\infty$,
$\dfrac{1}{q}=\dfrac{1}{p_{1}}+\dfrac{1}{p_{2}}$, $1<r<q$,
 $f_{1}\in
L^{p_{1}}(\mu)$, $f_{2}\in L^{p_{2}}(\mu)$, $b_{1}\in \text{RBMO}(\mu)$ and
$b_{2}\in \text{RBMO}(\mu)$.

When $l(B)\geq 1$, by $|f(x)|\leq
N_{\delta}f(x)$, Lemma \ref{lem2.1}, Lemma \ref{lem2.4}, Lemma \ref{lem2.5}, H\"{o}lder's inequality
and the boundedness of $M_{(\rho)}$ and $M_{r,(\rho)}$ for $\rho\geq
5$ and $q>r$, it follows
\begin{equation*}
\begin{split}
&||[b_{1},b_{2},T](f_{1},f_{2})||_{L^{q}(\mu)} \leq
||N_{\delta}([b_{1},b_{2},T](f_{1},f_{2}))||_{L^{q}(\mu)}
\leq C||M^{\sharp}_{\delta}([b_{1},b_{2},T](f_{1},f_{2}))||_{L^{q}(\mu)}\\
\leq &C||b_{1}||_{\ast}||b_{2}||_{\ast}||M_{r,(6)}(T(f_{1},f_{2}))||_{L^{q}(\mu)}
+C||b_{1}||_{\ast}||M_{r,(6)}([b_{2},T](f_{1},f_{2}))||_{L^{q}(\mu)}\\
&+C||b_{2}||_{\ast}||M_{r,(6)}([b_{1},T](f_{1},f_{2}))||_{L^{q}(\mu)}
+C||b_{1}||_{\ast}||b_{2}||_{\ast}||M_{p_{1},(5)}f_{1}M_{p_{2},(5)}f_{2}||_{L^{q}(\mu)}\\
&+C||b_{1}||_{\ast}||b_{2}||_{\ast}||M_{r,(6)}(T(f_{1}\chi_{\frac{6}{5}\tilde{B}},f_{2}\chi_{\frac{6}{5}\tilde{B}}))||_{L^{q}(\mu)}\\
&+C||b_{1}||_{\ast}||M_{r,(6)}([b_{2},T](f_{1}\chi_{\frac{6}{5}\tilde{B}},f_{2}\chi_{\frac{6}{5}\tilde{B}}))||_{L^{q}(\mu)}\\
&+C||b_{2}||_{\ast}||M_{r,(6)}([b_{1},T](f_{1}\chi_{\frac{6}{5}\tilde{B}},f_{2}\chi_{\frac{6}{5}\tilde{B}}))||_{L^{q}(\mu)}\\
\leq
&C||b_{1}||_{\ast}||b_{2}||_{\ast}||f_{1}||_{L^{p_{1}}(\mu)}||f_{2}||_{L^{p_{2}}(\mu)}
 +C||b_{1}||_{\ast}||([b_{2},T](f_{1},f_{2}))||_{L^{q}(\mu)}\\
&+C||b_{2}||_{\ast}||([b_{1},T](f_{1},f_{2}))||_{L^{q}(\mu)}\\
\leq
&C||b_{1}||_{\ast}||b_{2}||_{\ast}||f_{1}||_{L^{p_{1}}(\mu)}||f_{2}||_{L^{p_{2}}(\mu)}
 +C||b_{1}||_{\ast}||M^{\sharp}_{\delta}([b_{2},T](f_{1},f_{2}))||_{L^{q}(\mu)}\\
&+C||b_{2}||_{\ast}||M^{\sharp}_{\delta}([b_{1},T](f_{1},f_{2}))||_{L^{q}(\mu)}\\
\leq &C||b_{1}||_{\ast}||b_{2}||_{\ast}||f_{1}||_{L^{p_{1}}(\mu)}||f_{2}||_{L^{p_{2}}(\mu)}
+C||b_{1}||_{\ast}||M_{r,(6)}(T(f_{1},f_{2}))||_{L^{q}(\mu)}\\
&+C||b_{1}||_{\ast}||M_{p_{1},(5)}f_{1}M_{p_{2},(5)}f_{2}||_{L^{q}(\mu)}
+C||b_{2}||_{\ast}||M_{r,(6)}(T(f_{1},f_{2}))||_{L^{q}(\mu)}\\
&+C||b_{2}||_{\ast}||M_{p_{1},(5)}f_{1}M_{p_{2},(5)}f_{2}||_{L^{q}(\mu)}\\
\leq&C||b_{1}||_{\ast}||b_{2}||_{\ast}||f_{1}||_{L^{p_{1}}(\mu)}||f_{2}||_{L^{p_{2}}(\mu)}
+C||M_{p_{1},(5)}f_{1}||_{L^{p_{1}}(\mu)}||M_{p_{2},(5)}f_{2}||_{L^{p_{2}}(\mu)}\\
\leq&C||b_{1}||_{\ast}||b_{2}||_{\ast}||f_{1}||_{L^{p_{1}}(\mu)}||f_{2}||_{L^{p_{2}}(\mu)}.
\end{split}
\end{equation*}

When $0<l(B)<1$, using the above method, we also obtain the results of Theorem \ref{main-thm}. Thus the proof of Theorem \ref{main-thm} has been completed.
\end{proof}


%

{\bf Acknowledgements.} The authors would like to thank the referees for their very carefully reading and many valuable suggestions. The work was supported by Natural Science Foundation of China (No.11971026),  Natural Science Foundation of Education
Committee of Anhui Province (Nos.KJ2016A506, KJ2017A454) and Excellent Young Talents
Foundation of Anhui Province (Nos.GXYQ2017070, GXYQ2020049).

\end{document}